\documentclass[11pt,leqno]{amsart}
\usepackage{amsthm,amsfonts,amssymb,amsmath,oldgerm}
\usepackage{epsfig}
\numberwithin{equation}{section}
\usepackage[thinlines]{easybmat}
\usepackage{bbm}

\usepackage{hyperref}

\renewcommand\d{\partial}

\def\eps{\varepsilon }

\renewcommand\d{\partial}
\DeclareMathOperator{\dD}{d\!}

\renewcommand\d{\partial}

\newcommand\N{\mathbb N}
\newcommand\R{\mathbb R}

\def\eps{\varepsilon}

\newcommand\br{\begin{remark}}
\newcommand\er{\end{remark}}
\newcommand\bp{\begin{pmatrix}}
\newcommand\ep{\end{pmatrix}}
\newcommand\be{\begin{equation}}
\newcommand\ee{\end{equation}}
\newcommand\ba{\begin{equation}\begin{aligned}}
\newcommand\ea{\end{aligned}\end{equation}}

\newcommand{\bap}{\begin{app}}
\newcommand{\eap}{\end{app}}
\newcommand{\begs}{\begin{exams}}
\newcommand{\eegs}{\end{exams}}
\newcommand{\beg}{\begin{example}}
\newcommand{\eeg}{\end{exaplem}}
\newcommand{\bpr}{\begin{proposition}}
\newcommand{\epr}{\end{proposition}}
\newcommand{\bt}{\begin{theorem}}
\newcommand{\et}{\end{theorem}}
\newcommand{\bc}{\begin{corollary}}
\newcommand{\ec}{\end{corollary}}
\newcommand{\bl}{\begin{lemma}}
\newcommand{\el}{\end{lemma}}
\newcommand{\bd}{\begin{definition}}
\newcommand{\ed}{\end{definition}}
\newcommand{\brs}{\begin{remarks}}
\newcommand{\ers}{\end{remarks}}

\newtheorem{theo}{Theorem}[section]

\newtheorem{exams}[theo]{Examples}

\numberwithin{equation}{section}

\newcommand{\U }{\mathcal{U}}

\newcommand{\RR}{{\mathbb R}}

\newcommand{\const}{\text{\rm constant}}

\newtheorem{theorem}{Theorem}[section]
\newtheorem{proposition}[theorem]{Proposition}
\newtheorem{corollary}[theorem]{Corollary}
\newtheorem{lemma}[theorem]{Lemma}
\newtheorem{definition}[theorem]{Definition}

\newtheorem{example}[theorem]{Example}
\newtheorem{remark}[theorem]{Remark}

\newcommand{\beq}{\begin{equation}}
\newcommand{\eeq}{\end{equation}}

\title[Convective-wave solutions of the Richard-Gavrilyuk model]{
Convective-wave solutions of the Richard-Gavrilyuk model
for inclined shallow water flow }

\author{L.~Miguel Rodrigues}
\address{
Univ Rennes \& IUF, CNRS, IRMAR - UMR 6625, F-35000 Rennes, France}
\email{{\tt luis-miguel.rodrigues@univ-rennes1.fr}}
\thanks{Research of L.M.R. was partially supported by EPSRC grant no EP/R014604/1.}

\author{Zhao Yang}
\address{Academy of Mathematics and Systems Science, Chinese Academy of Sciences, Beijing
100190 China.}
\email{yangzhao@amss.ac.cn}
\thanks{Research of Z.Y. was partially supported by an IU COAS dissertation year fellowship}

\author{Kevin Zumbrun}
\address{Indiana University, Bloomington, IN 47405}
\email{kzumbrun@indiana.edu}
\thanks{Research of K.Z. was partially supported
under NSF grants no. DMS-1400555 and DMS-1700279}
\begin{document}

\maketitle

\begin{abstract}
We study for the Richard-Gavrilyuk model of inclined shallow water flow, an extension of the classical Saint Venant equations incorporating vorticity, the new feature of \emph{convective-wave} solutions analogous to contact discontinuitis in inviscid conservation laws. These are traveling waves for which fluid velocity is constant and equal to the speed of propagation of the wave, but fluid height and/or enstrophy (thus vorticity) varies. Together with hydraulic shocks, they play an important role in the structure of Riemann solutions.

\smallskip


\smallskip

\end{abstract}

\section{Introduction}\label{s:intro}
In this note we study new 
traveling-wave solutions, that we call \emph{convective waves}, 
of the recently-introduced Richard-Gavrilyuk model (RG) of inclined (incompressible) shallow-water flow \cite{R,RG1,RG2}, of a type not present in the classical Saint-Venant equations (SV) from which (RG) descends.

The Saint-Venant equations are the industry standard in hydroengineering applications such as dam or spillway design \cite{BM,JNRYZ}, having been used --- apparently successfully --- unchanged for nearly a century \cite{Je}. However, the phenomena they model are sufficiently complicated that the limits of their applicability are difficult to determine. See, for example, the discussion of roll wave stability for \eqref{sv} in \cite{JNRYZ}, showing the delicacy of that question. And, in the case of roll waves at least, it has been known since the experimental work of Brock \cite{Br1,Br2} that the (explicit) roll wave solutions of \eqref{sv} deviate in shape from experimentally observed profiles, exhibiting an ``overshoot'' phenomenon near shock discontinuities. This inconsistency has recently been resolved by Richard and Gavrilyuk \cite{R,RG1,RG2} by incorporating small-scale vorticity in the modeling, converting the Saint-Venant equations (SV) to the extended Richard-Gavrilyuk model (RG)- {\it the first such advance in nearly 100 years.} See Fig. \ref{figure1} showing experimental inaccuracy of Saint-Venant waves near breaking, pointed out in \cite{Br1,Br2} as compared to near-exact fit of (RG) roll waves.

\begin{figure}[htbp]
\begin{center}
\includegraphics[scale=.7]{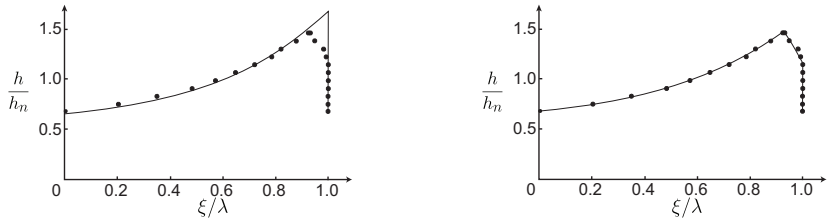}
\end{center}
\caption{Numerics vs. experiment. Left panel: Saint-Venant. Right panel: Richard-Gavrilyuk.    }
\label{figure1}
\end{figure}

The introduction of the (RG) model raises a number of interesting new questions: how does the existence theory for traveling waves differ between (SV) and (RG); can the rigorous stability theories developed in \cite{JNRYZ,YZ} for (SV) be adapted to the more complicated (RV)? and, most important, does (RV) predict new physical phenomena not captured by (SV)? These questions 
are 
addressed for hydraulic shocks and roll waves in \cite{RYZ} and \cite{RYZ2}. Here, we point out a new type of traveling wave occurring in (RG) but not (SV), of {\it convective-wave} solutions propagating with constant speed equal to the (everywhere constant) fluid velocity, but with fluid height and small-scale vorticity varying.

These may be understood as relaxation profiles for contact discontinuities of an associated equilibrium system, in the same way that hydraulic shock solutions are relaxation profiles for equilibrium shocks \cite{Bre,W,L,YZ,RYZ,SYZ}. Different from usual contact profiles, these are seen to be of degenerate type appearing in an infinite-dimensional family of possible nearby shapes. Evidently, this corresponds at linearized level to an infinite family of zero-eigenvalue modes, implying neutral stability at best, or orbital stability within this encompassing infinite-dimensional family. However, remarkably, we are able to factor out this degeneracy at the spectral level. Furthermore, for the subclass of such waves that are 
both asymptotically constant at infinity and piecewise smooth with a finite number\footnote{Possibly zero.} of discontinuitites,
we establish by a generalized Sturm-Liouville argument (similar to \cite{SYZ,SZ}) that their spectral stability is completely determined by the spectral stability of their limiting endstates, hence reduced to an explicit stability condition.

Numerical time-evolution experiments confirm this conclusion, showing that Riemann data is resolved into an asymptotic pattern consisting of a hydraulic shock plus a convective contact wave, as predicted by the associated equilibrium system. They also show that small initial perturbations of a given convective wave leads in large-time to {\it another} significantly different convective wave, in agreement with the above-mentioned infinite-dimensional degneracy. Proving this observed nonlinear stability remains an interesting open question; see discussion in Section~\ref{s:disc}.

\medskip

The content of the present contribution is organized as follows. In Section~\ref{s:rg} we introduce the (RG) model. Besides its main basic properties, we provide there some educated guesses about expected nonlinear dynamics, based on intuitive analogies with more standard relaxation systems. Section~\ref{s:rg} mixes elementary observations with formal and heuristic arguments. For comparison, it is preceded in Section~\ref{s:sv} by a similar discussion for the (SV) model. Rigorous mathematical analysis is contained in Sections~\ref{s:convective}, \ref{s:stab-smooth} and~\ref{s:discont}, where we investigate the structure of convective waves and their spectral stability. They contain our main theorems, Theorem~\ref{th:main} that elucidates spectral stability of convective waves with smooth profiles, Theorem~\ref{th:conv} that studies convective stabilization by spatial weights, and Theorem~\ref{th:disc} that extends the analysis to discontinuous profiles. In Section~\ref{s:num} we provide various numerical experiments, validating some of the proved facts, confirming some of our intuitions but also trailblazing in the wild. We conclude with a few perspectives in Section~\ref{s:disc}.

\bigskip

\noindent {\bf Acknowledgment:} The authors would like to warmly thank Sergey Gavrilyuk and Pascal Noble for enlightening discussions about the (RG) model. L.M.R. and Y.Z. also express their gratitude to Indiana University for its hospitality during part of the preparation of the present contribution. L.M.R. would like to thank the Isaac Newton Institute for Mathematical Sciences, Cambridge, for support and hospitality during the programme \emph{Dispersive hydrodynamics}.

\section{The Saint Venant equations}\label{s:sv}

The Saint-Venant equations take the form of a $2\times 2$ {\it hyperbolic relaxation system}
\ba\label{sv}
h_t + (hU)_x&=0,\\
(hU)_t 
+(hU^2+ p(h)  )_x
&=\hat gh-C_f|U|U,\\
\ea
where $h$ and $u$ denote fluid height and (vertically averaged) velocity at distance $x$ along an inclined ramp,
\be\label{piso}
p(h)=  g'\frac{h^2}{2},
\ee
and $g'=g\cos \theta$, $\hat g=g\sin \theta$, where $g$ is the gravitational constant, $\theta$ is the angle from horizontal of the ramp, and $C_f$ is a coefficient of turbulent bottom friction, modeled according to Ch\'ezy's law as proportional to velocity squared.
See, e.g., \cite{Je,Dr,BM,R,RG1,RG2}. 

The lefthand (first-order) side of \eqref{sv}--\eqref{piso} may be recognized as the equations of isentropic compressible gas dynamics with $\gamma$-law pressure, where $h$ plays the role of density and $\gamma=2$, from which we may deduce hyperbolicity \cite{Da,Bre}, with characteristics
\begin{align}\label{svchar}
\alpha_1&=U- a,&\alpha_2&=U+ a, 
\end{align}
where $a$ is sound speed, given by
\be\label{svsound}
a=\sqrt {p_h} =\sqrt{g'H_0}.
\ee

The formal equilibrium system obtained by setting the righthand (zero-order) side to zero is the scalar, Burgers-type equation 
\begin{align}\label{burgers}
h_t +q(h)_x&=0,&
\textrm{where }q(h):= \sqrt{\frac {\hat g h^3} {C_f}},
\end{align}
with characteristic speed
\be\label{echar}
\alpha_*=q'(h)=
\frac32 \sqrt{\frac {\hat g h} {C_f}},
\ee
where $u(h)= \sqrt{\frac {\hat g h} {C_f} }$ is determined by the equilibrium condition. In situations of hydrodynamic stability, or stability of constant-height equilibrium flow $(h,u)\equiv (H_0, U_0)$, $\hat g H_0=C_f U_0^2$, \eqref{burgers} is expected to approximately govern near-equilibrium behavior, in particular exhibiting ``hydraulic shock'', or ``bore'' type solutions \cite{Je,W,L,YZ,SYZ}. In situations of hydrodynamic instability, one observes, rather, pattern-formation and appearance of periodic ``roll wave'' solutions; see, e.g., \cite{C,Dr,Br1,Br2,BM,JNRYZ}.

As derived by Jeffreys \cite{Je,W}, the condition for hydrodynamic stability is
\be\label{hydro}
F:= \frac{U_0 }{a}< 2,
\ee
where 
$F$ is the {\it Froude number}, a dimensionless constant relating fluid velocity to sound speed.
Condition \eqref{hydro} corresponds to the {\it subcharacteristic} (or ``interlacing'') {\it condition} 
\be\label{sc}
\alpha_1<\alpha_*<\alpha_2
\ee
of Whitham \cite{W,L}, a standard necessary condition for hydrodynamic stability of relaxation systems. For, writing $q(H_0)$ as $q=H_0 u(H_0)$, where $u(h)= \sqrt{\frac {\hat g H_0} {C_f} }$, we have 
\begin{align}\label{alphastar}
\alpha_*&= \frac{dq}{dH_0}= U_0 + a_*,&\textrm{where   }
a_*=H_0 u'(H_0)= \frac{1}{2} U_0.
\end{align}
Comparing to \eqref{svchar}, we find that \eqref{sc} is equivalent to $a_*< a$, yielding \eqref{hydro}, independent of the choice of pressure function $p$ in \eqref{sv}. For the choice \eqref{piso} arising in the Saint Venant model (SV), $F$ is independent of height $H_0$ (or, equivalently, of $U_0=u(H_0)$), reducing to 
\be\label{Fsv}
F= \frac{U_0 }{\sqrt{g'H_0} }= \sqrt{\frac {\tan \theta}{C_f }}.
\ee

\section{The Richard-Gavrilyuk equations}\label{s:rg}
The Richard-Gavrilyuk equations are a $4\times 4$ relaxation system \cite[pp 383--384]{R} 
\ba\label{rg4}
h_t + (hU)_x&=0,\\
(hU)_t + (hU^2+ p )_x&=\hat gh-C|U|U,\\
\big(hE)_t + \big(hUE + Up)_x &= (\hat gh -C_f U|U|)U,\\ 
(h\varphi)_t + (h\varphi U)_x&=0,
\ea
with
\begin{align}\label{gaslaws}
p&= \frac12 g'h^2+ (\Phi+\varphi)h^3,&
E&=\frac12 U^2+e,&
e&=\frac12(g'h + (\Phi+\varphi)h^2),
\end{align}
where $h$ and $U$ are fluid height and velocity; $\Phi$ and $\varphi$ are large- and small-scale enstrophies associated with vorticity; $g'$, $\hat g$, $C_t$, $C_f$ are physical parameters; and
\be\label{C}
C=C_f \frac{\varphi}{\Phi+\varphi} + C_t \frac{\Phi}{\Phi+\varphi}.
\ee

For reference, we recall here the physical meaning of variables $\Phi$ and $\varphi$,
as given in \cite{RG1,RG2}. These are so-called \emph{enstrophies}, consisting of squared vorticity. They comprise a splitting of total enstrophy/vorticity, with $\varphi$ corresponding to small-scale vorticity near the bottom, and $\Phi$ to large-scale vorticity near shocks. For both roll waves and hydraulic shocks, $\varphi$ is necessarily constant along profiles.  For smooth hydraulic shock profiles, $\Phi$ is constant as well, in agreement with the physical derivation of $\Phi$ \cite{RYZ}; for 
roll waves\footnote{We choose to restrain the term \emph{roll waves} to periodic waves similar to the ones of (SV), in particular discontinuous, despite the fact that there are periodic traveling waves of convective type, including but not restricted to discontinuous ones.} 
and discontinuous hydraulic shock profiles, $\Phi$ builds up near the component subshock discontinuity \cite{RYZ,RYZ2} again in agreement with the derivation.

Let us point out that there are actually two slightly different Richard-Gavrilyuk models. We choose to work with the original one, from \cite{RG1}. The second version, from \cite{RG2}, is obtained from \eqref{rg4} by interchanging the roles of $C$ and $C_f$ and is presented in \cite{RG2} as leading to qualitatively and quantitatively similar numerical results, with a structure seemingly closer to the one of (SV).

\subsection{Reduced $3\times 3$ equations}\label{s:reduced}
When setting $\varphi\equiv \const$ in \eqref{rg4}, we may eliminate the fourth equation, obtaining the reduced $3\times 3$ Richard-Gavrilyuk (RG3) model
\ba\label{rg3}
h_t + (hU)_x&=0,\\
(hU)_t + (hU^2+ p )_x&=\hat gh-C|U|U,\\
\big(hE)_t + \big(hUE + Up)_x &= (\hat gh -C_f U|U|)U,\\ 
\ea
with $\varphi\equiv \const$ considered as an additional model parameter.

This gives a large subclass of interesting solutions, including in particular all traveling waves $(h,U,\Phi, \varphi)(t,x)=(\bar h,\bar U,\bar \Phi, \bar \varphi)(x-ct)$ with\footnote{Or more generally with $\bar U$ taking the value $c$ on a discrete subset of the domain of smoothness of wave profiles.} $\bar U\neq c$ --- notably, roll waves, and hydraulic shocks \cite{R,RG1,RG2,RYZ,RYZ2}. For, combining \eqref{rg4}(i) and (iv), we obtain for smooth solutions the convection equation
\be\label{varphieq}
\varphi_t+ u\varphi_x=0
\ee
in place of \eqref{rg4}(iv). Thus, for smooth portions of a traveling-wave solution, $(\bar U-c)\bar \varphi_x=0$, giving $\bar \varphi\equiv \const$ so long as $\bar U\neq c$. Similarly, the first and fourth Rankine-Hugoniot conditions for \eqref{rg4} at a discontinuity of speed $c$ are $c[h]=[hU]$ and $c[h\varphi]= [hU\varphi]$, where $[\cdot]$ denotes jump across the discontinuity. Combining these and using $h> 0$ yields
$[\varphi]=0$ so long as $U$ is not equal to $c$ on both sides of the discontinuity.
It follows that $\varphi$ remains constant also across discontinuities with speed $c$ not equal to fluid velocity $U$.

\subsubsection*{Relation to gas dynamics}\label{s:gas}

Just as the lefthand (first-order derivative) part of the Saint-Venant equations (SV) correspond to isentropic gas dynamics with pressure function of a polytropic gas law with $\gamma=2$, the lefthand (first-order derivative) part of the reduced (RG3) model \eqref{rg3} corresponds to full gas dynamics with pressure law
\be\label{palt}
p(h,e)= 2he - \frac12g'h^2,
\ee
obtained by using \eqref{gaslaws}(iii) to eliminate total enstrophy $\Phi +\varphi$ in \eqref{gaslaws}(i):
similar to but not exactly the nonisentropic polytropic gas law with $\gamma=2$ $p_{poly}(h,e)=2he$.

By this observation, we may read off the characteristics of the first-order (homogeneous) part of \eqref{rg3} using standard gas-dynamical formulae \cite{Da,Bre,BFZ} as 
\begin{align}\label{3chars}
\tilde \alpha_1&= U-\tilde a,&
\tilde \alpha_2&= U,&
\tilde \alpha_3&= U+\tilde a,
\end{align}
where $a$ is sound speed, given by
\be\label{a}
\tilde a=\sqrt {p_h +  \frac{p p_e}{h^2}}=\sqrt{6e-2g'h}= \sqrt{g'h+ 3(\Phi+\varphi)h^2}.
\ee

We note in passing that total enstrophy $S(h,e):=\Phi+\varphi= \frac {2e-g'h}{h^2}$ 
serves as a specific entropy for the first-order, gas-dynamical part of the equations
\cite{R,RG1,RG2},
satisfying for temperature $T= \frac{1}{S_e}=\frac{h^2}{2}$ the thermodynamic law $ \dD e=-p \dD\tau + T\dD S$, $\tau=1/h$, or
\begin{align}\label{therm}
\dD e&=\frac{p}{h^2} \dD h + T\dD S,&\textrm{that is } S_h&= -\frac{p}{h^2 }S_e,&
\textrm{or }e_h&=\frac{p}{h^2}.
\end{align}
This gives as a consequence, again by standard gas-dynamical facts \cite{Da,BFZ} that the first-order part of the equation for $S=(\Phi+\varphi)$ is the simple convection equation
\be\label{Scon}
S_t+US_x=0
\ee

\br\label{conrmk}
Written in terms of convectional derivatives $\dot g:=(\partial_t + U\partial_x)g$, the gas-dynamics equations become
\begin{align}\label{con1}
\dot h&=- hU_x,&
\dot U&=-  p_x/h,&
\dot e&=- (p/h)U_x,
\end{align}
whence $\dot S= S_h \dot h + S_e \dot e \equiv 0$ if and only if $S_h h + S_e p/h=0$, or \eqref{therm}. Thus, any convected function $S(h,e)$ that is monotone in $e$ may serve as a specific entropy for gas dynamics; we may easily verify the role of $(\Phi+\varphi)$ a posteriori by direct computation 
$(\partial_t+U\partial_x)(\Phi+\varphi)=0$. 
\er

\br\label{conrmk2}
Alternatively, as in \cite{R,RG1,RG2}, expressing $p=\hat p(h,S)=\frac{g'h^2}{2}+ Sh^3$ using \eqref{gaslaws}(i) and
rewriting \eqref{con1} in standard fashion as $\dot h =- hU_x$, $\dot U =-  \hat p_x/h$, $\dot S =0$, we obtain a decoupled $S$-equation and isentropic gas dynamics with pressure $\hat p(\cdot, S)$, yielding convective characteristics $0, \pm \tilde a$, with $\tilde a=\sqrt{ \hat p_h}$, hence (by \eqref{therm}(iii)) more directly recovering \eqref{a}.
Note that the sound speed $\tilde a=\sqrt{\hat p_h}$ for (RG) agrees with the sound speed $a=\sqrt{p_h}$ for  (SV) only in the \emph{zero-entropic}, or zero-vorticity case $S\equiv 0$ in which (SV) was originally derived.
\er

\subsubsection*{Equilibrium  system}\label{s:equilibrium}
The formal equilibrium system obtained by setting zero order derivative terms to zero is the same Burgers type equation \eqref{burgers}
as for (SV), with 
\be\label{eqrel}
(u(h), \Phi(h))= \Big(\sqrt{\frac {\hat g h} {C_f} }, 0\Big),
\ee
and the same characteristic speed $\alpha_*=q'(H_0)=\frac 32 \sqrt{ \frac{\hat g H_0}{C_f}}$ given in \eqref{alphastar}. As noted in \cite{R,RG1,RG2}, (strong) hydrodynamic stability, or (strong) spectral stability of constant equilibrium flows for \eqref{rg3} holds provided
\be\label{hydroRGF}
\tilde F:=\frac{U_0}{\tilde a_0}<2,
\ee
a generalization to (RG) of condition \eqref{hydro} of Jeffreys \cite{Je}, and 
\be\label{hydroRGC}
C_f\geq C_t.
\ee
See Appendix~\ref{app:split} for detailed computations. 

Here, condition \eqref{hydroRGF} corresponds, by the same argument that was used for (SV), to the subcharacteristic condition $\tilde \alpha_1<\alpha_*<\tilde \alpha_3$ of Whitham \cite{W}, while \eqref{hydroRGC} corresponds to dissipativity of the zero-order derivative forcing terms on the righthand side of \eqref{rg3} in the convective, entropy mode $S$. Specifically, writing the complete reduced, inhomogeneous equations \eqref{rg3} in convective derivative form, we obtain
\be\label{inhomrg3}
\dot h + hU_x =0,\quad
\dot U +  p_x/h= \frac{\hat g - C |U|U}{h},\quad
\dot S =\Big(1-\frac{\varphi}{S}\Big) \frac{ (C_t-C_f)|U|^3}{h^3},
\ee
and thus,  
where $\Phi\neq0$ and $U\neq 0$, $\dot S \leq 0$ if and only if $C_t \leq C_f$, with strict inequality unless $C_t=C_f$.

\subsection{Full $4\times 4$ system}\label{s:full}
Augmenting \eqref{rg3} with the simplified $\varphi$ equation \eqref{varphieq}, we find for smooth solutions that the full (RG) equations, written in convective derivative form,
are the reduced equations \eqref{inhomrg3} together with $\dot \varphi=0$, hence hyperbolic with characteristics 
\begin{align*}
\tilde \alpha&:= U-\tilde a,&0,&&0,&&U+\tilde a,
\end{align*}
where sound speed $\tilde a$ is again as in \eqref{a}. Moreover, the decoupled convective $\varphi$ mode is always neutrally stable, hence hydrodynamic stability is again equivalent to the conditions \eqref{hydroRGF}--\eqref{hydroRGC} for the reduced system \eqref{rg3}, {\it but is now at best nonstrict} in the sense that the dispersion relations for the linearization about a constant equilibrium state $(h_0,U_0, \Phi_0, \varphi_0)$ include the neutral, $\varphi$ mode $\lambda_4(k)=-iU_0k$, $k\in \R$.

\subsubsection*{Equilibrium system}\label{s:eq4}
Likewise, setting the righthand side (zero-order derivative part) of \eqref{rg4} to zero gives the same equilibrium relations \eqref{eqrel}
as for the reduced model \eqref{rg3}. However, with the addition of \eqref{rg4}(iv), the formal equilibrium model becomes now a $2\times 2$ {\it system}
\ba\label{eq2}
h_t + q(h)_x&0, \\
(h\varphi)_t +  (hu(h) \varphi)_x&=0, 
\ea
or, for smooth solutions, $ h_t + q(h)_x=0$, $\varphi_t +  u(h) \varphi_x=0$, with characteristics $\alpha_{*,1}= \frac{3}{2} u(h)$ and $\alpha_{2,*}=u(h)$ in characteristic modes $h$ and $\varphi$. And, since $\alpha_{*,2}=u(h)$ is {\it linearly degenerate} (independent of the associated mode $\varphi$), system \eqref{eq2}, besides Burgers shock and rarefaction waves in mode $h$, admits the new feature of {\it contact discontinuities} \cite{Bre,Da} in the mode $\varphi$, traveling with characteristic speed $c=\alpha_{*,2}=U_0$ equal to fluid velocity.

Indeed, one readily finds that solutions of a Riemann problem for \eqref{eq2}, joining arbitrary left and right states $(h_L, \varphi_L)$ and $(h_R, \varphi_R)$ consist of a contact discontinuity given by a jump in $\varphi$ from $\varphi_L$ to $\varphi_R$  with $h=h_L$ held fixed, and propagating at speed $c=u(h_L)=h_L^{1/2}$, followed by a scalar shock or rarefaction wave in $h$ with $\varphi=\varphi_R$ held fixed, and propagating with speed $\geq \alpha_*(h_L)> u(h_L)$ in the rarefaction case $h_L<h_R$ and $c=(q(h_R)-q(h_L))/(h_R-h_L)> (q(h_L)-q(0))/h_L= u(h_L)$ in the shock case $h_L \geq h_R$.

Thus, following the formalism of \cite{W,L}, in situations of hydrodynamic stability,  one might expect, at least for near-equilibrium (small data) flow, that long-time asymptotic behavior of (RG) should be governed by a regularized version of that of \eqref{eq2}, that is, a superposition of relaxation profiles for a Burgers-type shock (or rarefaction) in the genuinely nonlinear $h$-mode and a contact discontinuity in the linearly degenerate $\varphi$ mode, arranged in order of increasing speed so as to form a noninteracting pattern of waves connecting equilibrium states. See for example, the corresponding analyses for the $2n\times 2n$ Jin-Xin model in \cite{HPW,Zh} under the assumption of strict hydrodynamic stability. We note that, unlike shock profiles, the contact profiles as constructed in \cite{HPW,Zh} are not traveling waves, but approximately self-similar solutions with diffusive, error-function scaling.

In the present, degenerate case of {\it neutral hydrodynamic stability}, the conclusion is less clear. It is straightforward to see that the Riemann problem \cite{Bre,Da}
for \eqref{eq2} has a unique solution consisting of a contact discontinuity in $\varphi$, with $h$ held fixed, followed by a shock or rarefaction in $h$, with $\varphi$ held fixed. And, for the shock case relevant to hydrodynamic engineering, that height $h$ is decreasing from left to right, there is a unique relaxation profile, as shown in \cite{RYZ}, connecting equilibrium states with a common value of $\varphi$.

However, as we show below, there is a large class of traveling convective-wave solutions with $U\equiv c$ that are candidates for relaxation profiles of a contact discontinuity. Moreover, as the $\varphi$ equation is decoupled from the rest of the system, it lacks the ``effective diffusion'' present in \cite{HPW,Zh}, so that one cannot expect the type of diffusive contact profiles studied there. There is also the interesting question what occurs in the case of hydrodynamic {\it instability} $F_*>2$ in terms of the linearly degenerate $\varphi$ field.
We investigate these questions in the remainder of the paper.

\section{Convective traveling waves}\label{s:convective}
The above discussion in Section \ref{s:full} motivates the study of constant-speed solutions of \eqref{rg4} with $U\equiv c=\const$, in particular those connecting equilibrium states at $x=\pm \infty$.

Let us start by considering general solutions with $U\equiv c$, without regard to asymptotic states at infinity. Under this assumption, system \eqref{rg4} reduces for smooth solutions to $(h,\varphi)_t+ c(h,\varphi)_x=0$, together with 
\begin{align*}
\d_x\left(g'\frac{h^2}{2} + (\Phi+ \varphi)h^3\right)&=\hat gh-C|c|c\,,\\
\frac12\,h^3\,\left(\d_t\Phi+c\,\d_x\Phi\right)
+c\,\d_x\left(g'\frac{h^2}{2} + (\Phi+ \varphi)h^3\right)&=\left(\hat gh-C_f|c|c\right)\,c\,,
\end{align*}
which imply
\[
\frac12 h^3\,\left(\d_t\Phi+c\,\d_x\Phi\right)
\,=\,c\,(C_t-C_f)\,\frac{\Phi}{\Phi+\varphi}\,.
\]

From the latter one readily deduces the following lemma.
\bl
If $C_t\neq C_f$ and $c>0$, solutions with $U$ constant that are global and bounded (forward in time) satisfy
\[
\lim_{t\to\infty} \|\Phi(t,\cdot)\|_{L^\infty}=0\,.
\] 
Furthermore if $C_t> C_f$ such solutions satisfy $\Phi\equiv 0$ and are unstable.
\el

This motivates to restrict further to $(U,\Phi)\equiv (c,0)$. This amounts to 
\begin{align}\label{twave}
(h,U,\Phi, \varphi)&=(\bar h, \bar U, \bar \Phi, \bar \varphi)(x-ct)\,,&
\textrm{with }(\bar U, \bar \Phi)\equiv(c,0)\,, 
\end{align}
and
\be\label{scalarode}
\left(g'\frac{\bar h^2}{2} +  \bar\varphi \bar h^3\right)' = \hat g\bar h -C_f |c|c.
\ee
That is, (i) {\it any such smooth solution} is a traveling wave, and (ii) every solution of \eqref{scalarode} yields a corresponding smooth traveling-wave solution of \eqref{rg4}, convected with fluid velocity. We denote these as {\it convective-wave} solutions.

\subsection{Piecewise smooth solutions} We now generalize the discussion to (piecewise smooth) discontinuous solutions with $(U,\Phi)\equiv (c,0)$. The foregoing arguments extend directly to smooth parts of the solutions. In turn at a discontinuity traveling with speed $s$, we find that the associated Rankine-Hugoniot conditions reduce to
\begin{align}\label{redRH1}
(s-c)[h]&=0\,,&
(s-c)[h\varphi]&=0\,,
\end{align}
together with
\be\label{redRH2}
[p]=\left[g'\frac{h^2}{2}+\varphi h^3\right]=0.
\ee
That is, we find that $s=c$, i.e., discontinuities are likewise convected with fluid velocity, so that solutions are again traveling waves, and that \eqref{scalarode} again holds, now in distributional sense. Thus, {\it convective-wave solutions}, whether smooth or piecewise smooth, {\it are completely described by \eqref{scalarode}.}

\subsection{Asymptotically constant solutions}

We now specialize the discussion to convective-waves with asymptotic limits $(h_\pm,\varphi_\pm)$ at $\pm\infty$. Note that from \eqref{scalarode} stems
\[
h_\pm=h_0:= \frac{C_f c^2}{\hat g}\,.
\]

Setting $\delta:=\bar h-h_0$, we obtain finally
\be\label{fode}
\left(g'\frac{(h_0+\delta)^2}{2} +  \bar\varphi (h_0+\delta)^3\right)' = \hat g\delta.
\ee
For $\delta $ integrable but otherwise arbitrary\footnote{We deliberately omit to specify positivity constraints.} (smooth or not), this has nontrivial asymptotically-constant solutions, given by
\be\label{weird}
\left(g'\frac{(h_0+\delta)^2}{2} +  \bar\varphi (h_0+\delta)^3\right)(x) = \kappa+ \hat g \int_{-\infty}^x \delta(y)\dD y,
\ee
$\kappa$ an arbitrary constant.
The convective-wave solutions so constructed have limiting asymptotic states
$(h_0, c, 0,\varphi_-)$ and $(h_0, c, 0, \varphi_+)$, with $(\varphi_-,\varphi_+)$ arbitrarily tunable by adjusting $\kappa$ and $\int_{-\infty}^{+\infty}\delta(x)\dD x$.

\br\label{nosv}
For the Saint-Venant equations \eqref{sv}, solutions with $U\equiv c= \const$ must likewise be traveling waves moving with speed $c$, but now satisfying 
$(p(h))'= \hat g h- C_f c^2$ with $p(h)= g' h^2/2$, or
$ h'= \big(\frac{g'}{\hat g}- \frac{C_f c^2}{g'h}\big),  $ an ODE with no nontrivial bounded solutions other than the single unstable equlibrium $h_0=\frac{C_fc^2}{\hat g}$.
Thus, no such asymptotically constant $U\equiv c$ waves exist (except for the constant equilibrium). Likewise, for the reduced model \eqref{rg3}, with $\varphi$ constant, we obtain $(U,\Phi)\equiv (c, 0)$ and $(p(h))'= \hat g h- C_f c^2$, $p(h)=g'h^2/2 + \varphi h^3$, leading to $ h'= \big(\frac{\hat g h- C_f c^2}{g'h+ 3\varphi h^2}\big):  $
again an ODE with a single unstable equilibrium, hence no nontrivial asymptotically constant waves.
\er

\br\label{rk:varphi}
Traveling waves with $\bar U$ not constant, have $\bar\varphi$ constant. Therefore the role of convective waves in a possible large-time traveling-wave resolution is precisely to convey $\varphi$-variations. Consistently the present analysis shows that one may indeed find an infinite-dimensional family of convective-wave fronts connecting any equilibria differing only by their $\varphi$-component.
\er

\br\label{rk:Phizero}
When motivating the restriction to solutions with $\Phi\equiv0$, we have excluded the case $C_t=C_f$ as somehow exceptional. Indeed, it is straightforward to check that in this special case, the freedom in convective-wave profiles is even larger. For instance, one may pick $h_0$, $\delta$ and $\kappa$ as above, but also pick $\bar\Phi$ arbitrarily and obtain a convective-wave profile by solving in $\bar \varphi$
\[
\left(g'\frac{(h_0+\delta)^2}{2} +  (\bar\varphi+\bar \Phi) (h_0+\delta)^3\right)(x) = \kappa+ \hat g \int_{-\infty}^x \delta(y)\dD y\,.
\]
\er

\subsection{Periodic solutions and beyond}
Evidently, in the same way, we may construct spatially periodic solutions analogous to roll waves by integrating \eqref{fode} with $\delta$ periodic, and zero mean. 

Note that this includes cases when such periodic waves are smooth. In contrast, the arguments in Remark \ref{nosv} show that there are no such convective periodic traveling waves for \eqref{sv} or \eqref{rg3}.

One may generalize the present construction so as to unify it with the asymptotically-constant case by noticing that we only need $\delta$ and an anti-derivative of $\delta$ to be bounded so as to obtain a bounded convective-wave.

\br\label{rk:stab}
The freedom in the construction includes the possibility to prescribe any discrete set as the set of discontinuous points for $\bar h$, discontinuities in $\bar\varphi$ being included in those. See the related discussion in \cite{JNRYZ,DR2}. However we stress that in principle it could well be that among this tremendously huge number of convective-wave solutions only a few types are stable. To exemplify this possibility, we point out that it follows from the analysis in \cite{DR1,DR2} that such a dramatic reduction does occur for scalar balance laws. A consequence of our spectral analysis is that for (RG) such a reduction does not occur, at least for asymptotically constant waves with a finite number of discontinuities.
\er

\section{Spectral stability of smooth solutions}\label{s:stab-smooth}

We now investigate spectral stability of convective waves, starting with smooth asymptotically constant solutions. 

\br
Note that unlike what happens in more standard traveling-wave analyses smoothness and localization of wave profiles is not determined by profile equations but could be tuned arbitrarily by the prescription of $\delta=h-h_0$. Since we believe that the only important disctinction from the point of view of applications is whether $\delta$ is continuous or not, in the present section, devoted to the smooth case, we assume that $\delta\in\mathcal{C}^\infty$ and, in the asymptotically constant case, that $\delta$ and all its derivatives are exponentially localized. The reader interested in relaxing this assumption may adapt arguments in \cite{Pego-Weinstein} to the situation at hand.
\er

\subsection{Eigenvalue equations} 
To begin with we investigate the eigenvalue problem.

Linearizing about a convective-wave solution 
\begin{align*}
(h,U,\varphi,\Phi)(t,x)&=(\bar h,\bar U,\bar \varphi,\bar \Phi)(x-ct),&(\bar U,\bar\Phi)\equiv (c,0),
\end{align*}
we obtain a linear evolution whose eigenvalue equations are
\be\label{mateval}
\lambda A^0 W+ (AW)'=EW,
\ee
where $W=(h,U, \Phi, \varphi)$, 
\ba  
A^0&=\left(\begin{array}{cccc} 1 & 0 & 0 & 0\\[0.25em]
c & \bar{h} & 0 & 0\\[0.25em]
\frac{c^2}{2}+\frac{3\bar{\varphi} \bar{h}^2}{2}+g'\bar{h} & c\bar{h} & \frac{\bar{h}^3}{2} & \frac{\bar{h}^3}{2}\\[0.25em]
\bar{\varphi}  & 0 & 0 & \bar{h} \end{array}\right)\\[0.25em]
A^1&=\left(\begin{array}{cccc} c & \bar{h} & 0 & 0\\[0.25em]
c^2+3\bar{\varphi} \bar{h}^2+g'\bar{h} & 2c\bar{h} & \bar{h}^3 & \bar{h}^3\\[0.25em] \frac{c\left(c^2+9\bar{\varphi} \bar{h}^2+4g'\bar{h}\right)}{2} & \frac{\bar{h}\left(3c^2+3\bar{\varphi} \bar{h}^2+2g'\bar{h}\right)}{2} & \frac{3c\bar{h}^3}{2} & \frac{3c\bar{h}^3}{2}\\[0.25em]
c\bar{\varphi}  & \bar{h}\bar{\varphi}  & 0 & c\bar{h} \end{array}\right)\\[0.25em]
E&=\left(\begin{array}{cccc} 0 & 0 & 0 & 0\\[0.25em]
\hat{g} & -2C_fc & -\frac{c^2\left(C_T-C_f\right)}{\bar{\varphi} } & 0\\[0.25em]
c\hat{g} & \hat{g}\bar{h}-3C_fc^2 & 0 & 0\\ 0 & 0 & 0 & 0 \end{array}\right)\\[0.25em]
A=A^1-cA^0&=\left(\begin{array}{cccc} 0 & \bar{h} & 0 & 0\\[0.25em]
\bar{h}\left(g'+3\bar{h}\bar{\varphi} \right) & c\bar{h} & \bar{h}^3 & \bar{h}^3\\[0.25em] c\bar{h}\left(g'+3\bar{h}\bar{\varphi} \right) & \bar{h}\left(\frac{c^2}{2}+\frac{3\bar{\varphi} \bar{h}^2}{2}+g'\bar{h}\right) & ch^3 & c\bar{h}^3\\[0.25em]
0 & \bar{h}\bar{\varphi}  & 0 & 0 \end{array}\right)
\ea 
and $A$ has a kernel of dimension $2$.  Taking the inner product with the left kernel of $A$ thus gives two algebraic relations between the variables, by which we may reduce the eigenvalue problem to a $2\times 2$ ODE.

Specifically, subtracting $\bar{\varphi}$ times the first equation from the fourth equation in \eqref{mateval} yields
\[
\lambda \varphi + \bar \varphi_x U=0,
\]
which gives one of the algebraic relations.  From this we may solve for $\varphi$ in terms of $U$:
\be
\label{eqvarphi}
\varphi=-\lambda^{-1} \bar \varphi_x U.
\ee 
This corresponds to left zero-eigenvector $\ell_1=(-\bar{\varphi},0,0,1)$ of $A$.
Likewise, for $\ell_2$ equal to the other left zero-eigenvector of $A$,
we obtain a relation 
\[
\ell_2 (\lambda A^0 +A'-E)w= 0.
\]
Taking $\ell_2=\left(c^2-3\bar\varphi \bar h^2-2g'\bar{h} , -2c ,2 , 0 \right)$ and using \eqref{eqvarphi}, this equation yields
\be 
\label{eqPhi1}
\Phi=-\frac{2\bar\varphi \left(C_fc^2-\hat{g}\bar{h}+\bar{h}^3\bar{\varphi}_{x}+3\bar{h}^2\bar{h}_{x}\bar{\varphi} +g'\bar{h}\bar{h}_{x}\right)}{2(C_f-C_T)c^3+\bar{h}^3\lambda \bar{\varphi} }U
\ee 
or
\be 
\label{eqPhi2}
\Phi=-\frac{2\bar\varphi \left(C_fc^2-\hat{g}\bar{h}+\left(\bar{h}^3\bar{\varphi} +\frac{1}{2}g'\bar{h}^2\right)'\right)}{2(C_f-C_T)c^3+\bar{h}^3\lambda \bar{\varphi} }U
\ee 
By \eqref{scalarode}, \eqref{eqPhi2} implies 
\be\label{eqPhi}
\Phi\equiv0.
\ee

The full eigenvalue equations \eqref{mateval} thus reduce to the $2\times 2$ ODE
\be 
\left(\left[\begin{array}{cc} 0 & \bar{h}\\ \bar{h}\left(g'+3\bar{h}\bar{\varphi} \right) & c\bar{h}-\frac{1}{\lambda}\bar{h}^3\bar{\varphi}_x \end{array}\right]\left[\begin{array}{c} h\\ U \end{array}\right]\right)'=\left(\begin{array}{cc} -\lambda  & 0\\ \hat{g}-c\lambda  & -2C_fc-\bar{h}\lambda  \end{array}\right)\left[\begin{array}{c} h\\ U \end{array}\right],
\ee 
or
\ba 
\label{2by2}
&\left[\begin{array}{cc} 0 & \bar{h}\\ \bar{h}\left(g'+3\bar{h}\bar{\varphi} \right) & c\bar{h}-\frac{1}{\lambda}\bar{h}^3\bar{\varphi}_x \end{array}\right]\left[\begin{array}{c} h\\ U \end{array}\right]'\\
=&\left[\begin{array}{cc} -\lambda  & -\bar{h}_{x}\\ \hat{g}-g'\bar{h}_{x}-6\bar{\varphi}\bar{h}\bar{h}_{x}-3\bar{h}^2\bar{\varphi}_{x}-c\lambda  & -2C_fc-c\bar{h}_{x}+\frac{3}{\lambda}\bar{h}^2\bar{h}_x\bar{\varphi}_x+\frac{1}{\lambda}\bar{h}^3\bar{\varphi}_{xx}-\bar{h}\lambda  \end{array}\right]\left[\begin{array}{c} h\\ U \end{array}\right].
\ea 

Finally, to adapt Sturm-Liouville type arguments from \cite{SYZ}, we observe that equation \eqref{2by2} may for any $\lambda\neq 0$ be reduced to second-order scalar form
\begin{align}\label{Ueq_aux}
h&=\frac{(\bar{h}U)'}{-\lambda},\\
\label{Ueq}
U''&+f_1 U'+(f_2\lambda^2+f_3\lambda+f_4)U=0,
\end{align}
where 
\ba \label{fs}
&f_1=\frac{4\,\bar{\varphi}_{x}\,\bar{h}^2+12\,\bar{h}_{x}\,\bar{\varphi} \,\bar{h}-\hat{g}+3\,g'\,\bar{h}_{x}}{\bar{h}(g'+3\,\bar{h}\,\bar{\varphi}) },
&f_2&=-\frac{1}{\bar{h}(g'+3\,\bar{h}\,\bar{\varphi})}<0,\\
&f_3=-\frac{2\,C_f\,c}{\bar{h}^2\,\left(g'+3\,\bar{h}\,\bar{\varphi} \right)}<0,
& f_4&=\frac{(\bar{\varphi}\bar{h}^3)''+g'\,\bar{h}_{xx}\,\bar{h}+g'\,{\bar{h}_{x}}^2-\hat{g}\,\bar{h}_{x}}{\bar{h}^2\,\left(g'+3\,\bar{h}\,\bar{\varphi} \right)}=0
\ea 
where $f_4=0$ is obtained from differentiating \eqref{weird} twice. 

{\it System \eqref{Ueq} with \eqref{eqvarphi}, \eqref{eqPhi} and \eqref{Ueq_aux} is, evidently, equivalent to \eqref{mateval} for all $\lambda\neq 0$}.

\subsection{Removal of degenerate modes}\label{s:degen}

We now show how the computations of the former subsection may be put in an a more functional-analytic framework so as to yield corresponding reductions at the spectral level, that is, for resolvent problems, as required by abstract semigroup theory.

The linearized dynamics obeys 
\be\label{eq:linearized}
A_0\d_t W+\d_x(A\,W)\,=\,E\,W
\ee
with $A_0$, $A$ and $E$ as above. By applying the invertible 
\[
\bp
\begin{matrix} 1&0&0&0\\
-c\,\bar{h}^{-1}&\bar{h}^{-1}&0&0\end{matrix}\\
\bar{h}^{-3}\ell_2-\bar{h}^{-1}\ell_1\\
\bar{h}^{-1}\ell_1
\ep
\]
one shows that \eqref{eq:linearized} is equivalently written as
\[
\d_t\bp h\\U\\\Phi\\\varphi\ep\,=\,L\bp h\\U\\\Phi\\\varphi\ep
\]
where 
\[
L\bp h\\U\\\Phi\\\varphi\ep
\,=\,\bp
-A_{red}\d_x\bp h\\U\ep
+E_{red}\bp h\\U\ep
+\bp 0\\-\bar{h}^{-1}\,\d_x\left(\bar{h}^3\,(\Phi+\varphi)\right)
-\frac{c^2\left(C_T-C_f\right)}{\bar{\varphi}\bar{h}}\Phi
\ep\\
-\frac{2(C_f-C_T)c^3}{\bar{h}^3\bar{\varphi}}\Phi\\
-\bar{\varphi}_x\,U
\ep
\]
with
\begin{align*}
A_{red}&=\bp0&\bar{h}\\
g'+3\bar{h}\bar{\varphi}&0\ep\,,&
E_{red}&=\bp 0&-\bar{h}_x\\
\bar{h}^{-1}\,
\left(\hat{g}-\d_x\left(\bar{h}\left(g'+3\bar{h}\bar{\varphi} \right)\right)\right)
&-2C_fc\,\bar{h}^{-1}\ep\,.
\end{align*}
The formal operator $L$ may be turned into a densely-defined closed operator on a Banach space $X$, with domain $D$, for various pairs $(X,D)$ of the form $X=X_0\times X_1$, $D=X_1\times X_1$, including those arising from $X_0=(W^{k,p}(\R))^2$ and $X_1=(W^{k+1,p}(\R))^2$ for any $(k,p)\in\N\times [1,\infty)$, or from $X_0=(BUC^{k}(\R))^2$ and $X_1=(BUC^{k+1}(\R))^2$ for any $k\in\N$ (where $BUC^\ell$ denotes functions whose derivatives up to order $\ell$ are bounded and uniformly continuous).

By definition, $\lambda$ does not belong to the spectrum of $L$ if and only if for any $F\in X$ there exists a unique $W\in D$ such that $(\lambda-L)W=F$, and the corresponding solution operator $(\lambda-L)^{-1}$ is bounded on $X$. 

\begin{definition}\label{def:spec}
An asymptotically-constant smooth convective-wave will be called spectrally stable (on the functional space $X$) provided that the spectrum of the corresponding $L$ is included in $\{\,\lambda\,;\,\, \Re(\lambda)\leq 0\,\}$. It is called strongly spectrally stable provided that the foregoing spectrum is included in $\{\,\lambda\,;\,\, \Re(\lambda)< 0\,\}\,\cup\,\{0\}$.
\end{definition}

The definition of strong spectral stability is motivated by the fact that in any case $0$ belongs to the spectrum. 

For the above-mentioned functional framework, one checks readily that $\lambda$ does not belong to the spectrum of $L$, if and only if $\lambda\neq0$,
\be\label{eq:spec-algebraic}
\lambda\notin -\overline{\frac{2(C_f-C_T)c^3}{\bar{h}^3\bar{\varphi}}(\R)}
\ee
and for any $G\in X_0$, there exists a unique $V\in X_1$ such that $L_{red}(\lambda)V=G$ with a solution operator $(L_{red}(\lambda))^{-1}$ bounded on $X_0$, where
\begin{align*}
&L_{red}(\lambda)\bp \tau\\U\ep
:= A_{red}\d_x\bp \tau\\U\ep
+\bp 1&0\\0&\lambda\ep(\lambda-E_{red})\bp \lambda^{-1}&0\\0&1\ep\bp \tau\\U\ep
+\bp 0\\-\bar{h}^{-1}\,\d_x\left(\bar{h}^3\,\bar{\varphi}_x\,U\right)\ep\\
&=\bp1&\d_x(\bar{h}\cdot)\\
\d_x\left(\left(g'+3\bar{h}\bar{\varphi} \right)\cdot\right)
-\bar{h}^{-1}\,
\left(\hat{g}-\bar{h}_x\left(g'+3\bar{h}\bar{\varphi} \right)\right)
&\lambda^2+\lambda\,2C_fc\,\bar{h}^{-1}
-\bar{h}^{-1}\,\d_x\left(\bar{h}^3\,\bar{\varphi}_x\,\cdot\right)\ep\bp \tau\\U\ep.
\end{align*}

This achieves the reduction from the original spectrum problem to the invertibility of a reduced operator $L_{red}(\lambda)$, whose kernel equation is precisely \eqref{2by2} for $(\tau,U)=(\lambda\,h,U)$. Note that, as the original $\lambda-L$, $L_{red}(\lambda)$ is a closed densely defined operator depending analytically on $\lambda$. 

A few remarks are in order to highlight what is the nature of the gain when going from $\lambda-L$ to $L_{red}(\lambda)$. 

\br
The original operator $L$ combines purely algebraic parts, that is, parts given as multiplication operators, with differential parts. The key point is that multiplication operators are Fredholm of index $0$ only when they are invertible so that the spectrum of a multiplication operator is reduced to its essential spectrum. Moreover the latter is determined by all the values of the function by which we multiply\footnote{See for instance \eqref{eq:spec-algebraic}.}, instead of being read on asymptotic limits at $\pm\infty$. This is this kind of degeneracy that is responsible for the fact that there is an uncountable family of zero eigenmodes of \eqref{mateval}, obtained through $(U,\Phi)\equiv (0,0)$ and
\[
( \bar h(g'+3\bar h\bar \varphi)h + \bar h^3 \varphi)'=\hat g h,
\]
the linear equivalent of the nonlinear solutions determined by $(\bar U,\bar \Phi)\equiv (c,0)$ and \eqref{weird}. By factoring out the algebraic part, we effectively remove this kind of degeneracy. However, it must be remembered that the full system does possess these degenerate modes, a fact with implications for nonlinear stability and asymptotic behavior ; see Section~\ref{s:disc} for further discussion. The reduced operators $L_{red}(\lambda)$ are one-dimensional\footnote{This is crucial since first-order operators possess elliptic properties only in dimension $1$.} non-characteristic differential operators with asymptotic limits at $\pm\infty$ (reached sufficiently fast) so that, as we detail below, in a large region of the spectral plane --- purely determined by asymptotic limits ---, these operators are Fredholm of index $0$, thus their invertibility is equivalent to their one-to-one character, hence to \eqref{Ueq} possessing no non trivial solution. Incidentally we point out that the failure of the non-characteristic requirement would bring other kinds of degeneracies; see the detailed discussion in \cite{DR2}. 
\er

\br
The crucial part of the reduction is the elimination of the algebraic parts. But we also had to perform a suitable scaling, including $h\mapsto \lambda h=:\tau$, so as to preserve regularity near $\lambda=0$. This scaling shares similarities with classical ``flux-type'' transformations, considered for instance in \cite{PZ}, and that may be thought as spectral counterparts to the widely-used anti-derivative trick --- dating back at least to \cite{Matsumura-Nishihara,Goodman_bis} --- introduced to remove the (non-degenerate) eigenvalue $0$. For more advanced, multi-dimensional, versions of the ``flux-type'' transformations we refer the reader to \cite{HLyZ2,BHLyZ}. More generally, the need for such types of scaling also arise when connecting at spectral level Eulerian and mass-Lagrangian formulations or quantum systems and their hydrodynamic formulations ; see \cite[Section~5.2]{BMR} on the former and \cite[Section~3.2]{Audiard-Rodrigues} on the latter.
\er

\subsection{Standard spectral stability} \label{s:smooth_as_const}

We now investigate invertibility properties of $L_{red}(\lambda)$. The goal is in a relevant region of the spectral plane to reduce it to invertibility properties of its spatial asymptotic limits
\begin{align*}
L_{red}^{\pm}(\lambda)\bp \tau\\U\ep
&:=\bp1&h_0\d_x\\
\left(g'+3h_0\varphi_{\pm}\right)\d_x
-h_0^{-1}\,\hat{g}
&\lambda^2+\lambda\,2C_fc\,h_0^{-1}\ep\bp \tau\\U\ep\,.
\end{align*}

Our strategy uses classical arguments to show that this reduction holds when \eqref{Ueq} possesses no non trivial solution, and then adapts arguments from \cite{SYZ} to show that the latter does hold in the spectral region of interest. On the former classical arguments we give little detail and rather refer the reader to the already classical \cite{Z-H,Z3,MZ,Sandstede,KapitulaPromislow-stability} for detailed comprehensive exposition and to the recent \cite{Blochas-Rodrigues} for a self-contained worked-out case that could hopefully be used as a gentle entering gate.

As a preliminary we point out that the invertibility of $L_{red}^{\pm}(\lambda)$ for any $\lambda$ such that $\Re(\lambda)>0$ (respectively for any $\lambda\neq0$ such that $\Re(\lambda)\geq0$) is equivalent to
\begin{align*}
F_{\pm}:=\sqrt{\frac{\hat{g}}{C_f(g'+3h_0\varphi_\pm)}}&\leq2\,,&
(\textrm{resp. } F_{\pm}<2)\,.
\end{align*}
The foregoing claim follows from the computations involved in the stability analysis for constant states, already worked out in \cite{R,RG1,RG2}, and provided in Appendix~\ref{app:split}. 

\begin{theorem}\label{th:main}
Consider an asymptotically-constant smooth profile with limiting values $(h_0, c, 0,\varphi_-)$ and $(h_0, c, 0, \varphi_+)$, $c>0$, reached exponentially fast. This wave is spectrally stable if and only if
\begin{align*}
C_f&\geq C_t\,,&
F_+&\leq 2\,,&
F_-&\leq 2\,.
\end{align*}
and it is strongly spectrally stable if and only if
\begin{align}\label{hydrodynamicstability}
C_f&\geq C_t\,,&
F_+&<2\,,&
F_-&<2\,.
\end{align}
\end{theorem}

\begin{proof}
The role of the condition on $(C_f,C_t)$ has been elucidated in the foregoing subsection. The necessity of the conditions on $F_+$ and $F_-$ stems directly from the classical fact that invertibility of both $L_{red}^{\pm}(\lambda)$ is necessary to invertibility of $L_{red}(\lambda)$. The foregoing claim is relatively easy to prove by building quasi-modes ; see Lemma~2 in the Appendix to \cite[Chapter~5]{He}, or Proposition~2.1 in \cite[Section~2.1]{DR2}. Knowing that $L_{red}^{\pm}(\lambda)$ are invertible when $\Re(\lambda)\geq 0$, $\lambda\neq0$, it is classical to deduce that for such $\lambda$, $L_{red}(\lambda)$ is invertible if and only if it is one-to-one. However the proof of the latter is longer and more involved and we simply refer to the already quoted literature. Let us only mention that the classical proof builds an inverse from Green functions, and that those are obtained by gluing together solutions to the differential equation encoding the kernel equation for $L_{red}(\lambda)$, considered on half-lines. Hence the strong connection between the structure of the kernel equation and invertibility.

We also observe that the sufficiency of the condition for spectral stability may be derived from the sufficiency of the condition for strong spectral stability through a limiting argument that we omit.

Thus, the remaining, and only non classical, task is to prove that when $C_f\geq C_t$, $F_+<2$, and $F_-<2$, and $\Re(\lambda)\geq 0$, $\lambda\neq0$, the kernel of $L_{red}(\lambda)$ is reduced to the null function, or, in other words, that \eqref{Ueq} possesses no non trivial solution (in the relevant functional space). To prove this, we modify the strategy from \cite{SYZ}.

The argument is closed by homotopy. For $\epsilon\in[0,1]$, consider the convective-wave $(\bar{h}_\epsilon,\bar{U}_\epsilon,\bar{\Phi}_\epsilon,\bar{\varphi}_\epsilon)$ built from \eqref{weird} with the same $(c,h_0,\kappa)$ but $\delta$ replaced with $\epsilon\,\delta$. At $\epsilon=0$, the wave is constant with $\bar{h}\equiv h_0$, $\bar{\varphi}_\epsilon\equiv\varphi_-$. The situation is favorable at $\epsilon=0$ thanks to $F_-<2$, and Froude conditions are seen to hold for any $\epsilon\in [0,1]$ by monotonicity in $\varphi_+$ of the criterion. Our task is to show that there is no transition from stability to instability as we vary $\epsilon$. It follows from general theory that transitions may occur only in two ways:
\begin{enumerate}
\item At some $\epsilon_0$ there is a $\lambda\in i\RR$, $\lambda\neq0$, such that the $\epsilon_0$-version of \eqref{Ueq} possesses a non trivial solution (in the relevant functional space).
\item For some $\epsilon_0$, for any $\epsilon>\epsilon_0$ sufficiently close to $\epsilon_0$, there exists $\lambda_\epsilon$ converging to $0$ when $\epsilon$ goes to $\epsilon_0$ such that $\Re(\lambda_\epsilon)>0$ and the corresponding $\epsilon$-version of \eqref{Ueq} possesses a non trivial solution (in the relevant functional space).
\end{enumerate}

We begin by excluding the first kind of transition. For the sake of readability, we omit to mark the $\epsilon_0$ dependence. The first observation is that under the present assumptions, exponential dichotomy holds: when $\Re(\lambda)\geq0$, $\lambda\neq0$, solutions to the differential equation \eqref{Ueq} that do not blow exponentially decay exponentially, with rates matching decaying rates near $\pm\infty$ of asymptotic equations
\begin{align*}
U''-\frac{\hat{g}}{\bar{h}(g'+3\,h_0\,\varphi_\pm)} U'
-\lambda\left(\frac{1}{h_0(g'+3\,h_0\,\varphi_\pm)}\lambda
+\frac{2\,C_f\,c}{h_0^2\,\left(g'+3\,h_0\,\varphi_\pm\right)}\right)U=0\,.
\end{align*}
This implies that any such solution provides via the Liouville-type transformation 
\be\label{Ltrans}
 w(x)=e^{\frac{1}{2}\int_0^x f_1(y)\dD y}U(x),
\ee
an element $w$ of the kernel of the operator $\mathcal{L}(\lambda)$ acting on $L^2$ with domain $H^2$ through
 \be 
 \label{weq}
\mathcal{L}(\lambda)(w):=w''+(f_2\lambda^2+f_3\lambda+f_4-\frac{1}{4}f_1^2-\frac{1}{2}f_1')w\,.
 \ee
To conclude this part, as in \cite[Lemma~3.1]{SYZ}, we derive from $f_3<0$ that when $\lambda\in i\RR$, $\lambda\neq0$, the kernel of $\mathcal{L}(\lambda)$ is trivial. Namely, taking the $L^2$ inner product of $i\,w$ against equation $\mathcal{L}(\lambda)(w)=0$ for $\lambda=i\alpha$, $\alpha\in\R^*$, we obtain $0=\langle w, f_3 w\rangle$, hence $w=0$.
 
It only remains to exclude the second kind of transition. Here we depart from \cite{SYZ}. Our first observation is that similarly classical arguments, building on limits of spatial decay rates, show that this transition may only occur if at $\epsilon_0$, the $\epsilon$-version of \eqref{Ueq} possesses at $\lambda=0$ a non zero solution $U$, bounded near $+\infty$ and decaying exponentially near $-\infty$. We now exclude this possibility, and again omit to mark any dependence on $\epsilon_0$ when doing so. Since $f_1(+\infty)>0$, the only bounded solutions to \eqref{Ueq} with $\lambda=0$, that is, to
\be\label{0eq}
U''+f_1 U'=0\,,
\ee
are solutions with $U'\equiv0$, and those do not decay to $0$ at $-\infty$ if they are not constantly zero. Hence the conclusion of this part and of the proof.
\end{proof}

\br\label{blockrmk}
For comparison, we mention that in \cite{SYZ} the argument blocking the emergence of eigenvalues through $\lambda=0$ builds on the fact that there is a non trivial element in the kernel but that its sign combined with Sturm-Liouville theory ensures that the multiplicity of $\lambda=0$ is fixed (equal to $1$), hence no crossing is possible. 
\er

\subsection{Convective spectral stability}\label{s:consubs}

Let us observe that in the foregoing subsection we have not specified which of the functional spaces introduced in \eqref{s:degen} we were picking. The reason is that this choice is immaterial to conclusions of Theorem~\ref{th:main}. The conclusion would be dramatically different if we were allowing to introduce spatial weights among possible Sobolev norms.

With this in mind, we revisit briefly and at a deeper level the hydrodynamic stability conditions $F_+\leq 2$, $F_-\leq 2$, and its role in spectral stability. Though this condition is clearly necessary in standard Sobolev norms, it is also well-known that stability properties may sometimes be changed by working in exponentially weighted norms, taking account of the fact that instabilities may be convected into a traveling wave and stabilized by near-field dynamics. The upshot is that a wave may be stable with respect to sufficiently exponentially localized perturbations even when it is unstable with respect to perturbations in standard norms, a phenomenon known as {\it convective stability} or {\it convective stabilization}. On this topic, we refer to \cite{Sat} for pioneering work, to \cite[Chapter~3]{KapitulaPromislow-stability} for a comprehensive discussion at spectral level and to \cite{GR,FRYZ,Blochas} for convective analogs to \cite{DR1,DR2,YZ,SYZ,Blochas-Rodrigues}.

In order to state a convective analog to Theorem~\ref{th:main}, we extend Definition~\ref{def:spec}, to \emph{convective} spectral stability when spectral stability is obtained in a weighted topology with a smooth weight lower-bounded away from zero, and to \emph{extended convective} spectral stability when it is met for some smooth positive weight.

The weight $e^{\frac{1}{2}\int_0^x f_1(y)\dD y}$, already used in \eqref{Ltrans}, ensures invertibility properties for $L_{red}^\pm(\lambda)$ for any $\lambda\neq 0$ such that $\Re(\lambda)\geq0$ independently of the position of $F_\pm$ with respect to $2$. Once this is done the rest of the proof of Theorem~\ref{th:main} applies almost without change. Note that the above weight goes to zero as $x\to\infty$ and an inspection of spatial decay at $+\infty$ shows that this cannot be fixed by any other lower-bounded weight ; see Appendix~\ref{app:split} for some details. However $e^{\frac{1}{2}\int_0^x \chi(y)\,f_1(y)\dD y}$ with $\chi$ smooth, equal to $1$ in a neighborhood of $-\infty$ and to $0$ in a neighborhood of $+\infty$, provides a lower-bounded weight that restores invertibility of $L_{red}^-(\lambda)$ in the zone of interest. 

The foregoing discussion leads to the following theorem.

\begin{theorem}\label{th:conv}
Consider an asymptotically-constant smooth profile with limiting values $(h_0, c, 0,\varphi_-)$ and $(h_0, c, 0, \varphi_+)$, $c>0$, reached exponentially fast.
\begin{enumerate}
\item Its extended convective spectral stability requires $C_f\geq C_t$, and this is sufficient to obtain strong extended convective spectral stability.
\item Its convective spectral stability is equivalent to $C_f\geq C_t$ and $F_+\leq 2$, and its strong convective spectral stability is equivalent to $C_f\geq C_t$ and $F_+< 2$.
\end{enumerate}
\end{theorem}

\br
The outcome of the analysis is consistent with the classical rule of thumb that only instabilities traveling towards the wave may be stabilized with a classical weight. To carry out this consistent check, we first recall \cite{MZ} that transition to hydrodynamic instability is closely related to the Chapman-Enskogg expansion \eqref{eq2}, with instabilities occurring in neutral modes corresponding to the characteristic modes of \eqref{eq2}. As the $\varphi$ mode decouples to all orders, it remains always neutral, hence strict, or exponential, instability should it arise occurs in the $h$-mode, convected with characteristic speed $\alpha^2_*$ strictly greater than the speed $\alpha^1_*=U$ of the $\varphi$ mode. Thus, for a {\it hydrodynamically unstable state to the left} of an asymptotically-constant convective wave solution,  instabilities are convected inward toward the wave, at speed greater than the speed $c=U$ of the wave. Such instabilities are thus subject to convective stabilization, recovering stability with respect to sufficiently localized perturbations. For a {\it hydrodynamically unstable state to the right} of an asymptotically-constant convective wave solution,  on the other hand, instabilities are convected outward from the wave, and cannot be so stabilized (in a classical way).
\er

There is some hope to directly use convective stability at the nonlinear level, since involved topologies behaves nicely with respect to nonlinear estimates. Nonlinear results mentioned above fit in this frame.

In turn, our motivation to also investigate extended convective stabilization and our expectations concerning its role in the nonlinear dynamics is somewhat subtle. Generally speaking, we expect that a wave that could be stabilized in an extended convective sense could enter as a block in a stable multi-wave pattern. For the present case, we are particularly interested in $2$-wave patterns, with the second (rightmost) wave (corresponding to the $h$ mode of \eqref{eq2}) being a {\it Lax shock}. In particular, then, characteristics are propagated into the shock from either
side. This observation will be important in Section \ref{s:num} below, in interpreting the results of our numerical time-evolution experiments.

Let us stress however that all our present discussion about nonlinear dynamics is somewhat speculative. Indeed, due to the presence of an infinitely degenerate neutrally stable mode, even in the smooth stable case, the present problem does not fit under the application of any result that we know of converting spectral stability into nonlinear stability. 

\subsection{Stability of periodic solutions} \label{s:periodic}

Among the wide variety of smooth convective waves, besides the asymptotically-constant ones, those with a periodic profile also fit in a reasonably developed spectral stability framework. The latter is much more recent though. See for instance \cite{G,OZ1,JNRZ13,R_linearKdV} for a few significant contributions to periodic-wave nonlinear and spectral analyses, and \cite{R_HDR,R_Roscoff} for some detailed accounts. We restrain from tackling such a stability analysis, but we would like to point out why the asymptotically-constant analysis is not readily adaptable to the periodic problem. 

A significant part of the asymptotically-constant study is deduced from an energy estimate carried out for the self-adjoint operator in \eqref{weq} obtained through the Liouville transformation $w=e^{\frac{1}{2}\int_0^x f_1(y)\dD y}U$ of \eqref{Ltrans}. Yet, in the periodic case, such a transformation is not available, unless $f_1$ is mean-free over a period, since otherwise $e^{\frac{1}{2}\int_0^x f_1(y)\dD y}$ is not periodic and grows without bound near one infinity, decays to zero near the other infinity. 

We point out that, for discontinuous periodic waves, the stability framework of \cite{JNRYZ,DR2} could also be adapted to the present context. Yet we shall not carry out this analysis either. 

Thus, in the periodic case, we do not obtain analytic results either for smooth or discontinuous profiles.

\section{Stability of discontinuous solutions}\label{s:discont}

We now consider briefly the case of piecewise-smooth convective waves, with a finite number of discontinuities, that are asymptotically constant. We prove that these waves are also spectrally stable when the limiting endstates are so. Our proof proceeds by a limiting argument from the stability of asymptotically-constant smooth waves. By taking a limit we lose the strong stability part of Theorem~\ref{th:main}.

The choice of such an argument also reflects that we have not been able to extend directly the arguments of the smooth analysis to the discontinuous case. To provide some insights on how these fail, let us give some details on what we are able to obtain in this way. The reduction to a nice eigenvalue problem and the proof of the impossibility that transitions to instability occur through the emergence of eigenvalues from $\lambda=0$ may indeed be adapted with little changes. The important part that we are not able to extend is the exclusion of transitions to instability through nonzero purely imaginary eigenvalues. In the smooth case, this eventually relies on an energy\footnote{There is some freedom in the energy estimate used but we have failed to adapt any of those.} estimate and discontinuities introduce, in the estimate, jump contributions that we have not been able to manage. Thus, by using real symmetry, the extension of the smooth-case proof only yields that transition to instability may only occur through the passage of pairs of nonzero complex conjugates eigenvalues. We are also able to extend the argument used to show the impossibility of a passage through $0$ so as to compute explicitly an instability index, counting modulo $2$ the number of eigenvalues on $(0,\infty)$; see \cite{GZ,Pego-Weinstein,BJRZ,BMR,JNRYZ} for similar computations. The upshot of the latter --- consistent with the rest of the present arguments --- is that this number is always even\footnote{We recall that we show by different arguments that this number is always $0$.}.

The rest of the section is devoted to the proof of the claimed spectral stability formalized in the following theorem.

\begin{theorem}\label{th:disc}
Consider an asymptotically-constant piecewise-smooth profile with limiting values at infinities $(h_0, c, 0,\varphi_-)$ and $(h_0, c, 0, \varphi_+)$, $c>0$, reached exponentially fast, and exhibiting a finite number of discontinuities. This wave is spectrally stable if and only if
\begin{align*}
C_f&\geq C_t\,,&
F_+&\leq 2\,,&
F_-&\leq 2\,.
\end{align*}
\end{theorem}

Some details on the notion of spectral stability used in the foregoing statement are given in the following subsection.

\subsection{The discontinuous eigenvalue problem}

For the sake of simplicity, we prove Theorem~\ref{th:disc} only for profiles with a single discontinuity, which without loss of generality we fix at $0$. We stress however that the adaptation to the general case is mostly notational.

Following \cite{JNRYZ,DR2}, we linearize in $(h,U,\varphi,\Phi,\psi)$ the system obtained from inserting the solution \emph{ansatz} 
\[
(t,x)\longmapsto
(\bar h,\bar U,\bar \varphi,\bar \Phi)(x-ct-\psi(t))
+(h,U,\varphi,\Phi)(t,x-ct-\psi(t))
\]
in \eqref{rg4}, with $(\bar h,\bar U,\bar \varphi,\bar \Phi)$ and $(h,U,\varphi,\Phi)$ smooth on $\R^*$. This yields interior equations, on $\R^*$, 
\[
A^0\d_tW+\d_x(AW)=EW,
\]
for $W=(h,U, \Phi, \varphi)-\psi\,(\bar h_x,\bar U_x,\bar \varphi_x,\bar \Phi_x)$, where $A^0$, $A$ and $E$ are as in \eqref{mateval}, supplemented with linearized Rankine-Hugoniot conditions. Concerning the latter we first observe that the linearized jump conditions associated with the first and fourth equations of \eqref{rg4} reduce to
\begin{align*}
[U]&=0\,,&\psi'=U(\cdot,0)\,,
\end{align*}
provided that $\bar\varphi$ does jump, which we will assume from now on. Once those are enforced the two remaining jump conditions reduce to a single equation
\[
[(g'\bar h+3\bar\varphi\bar h^2)\,h+\bar h^3\,(\Phi+\varphi)]\,=\,0\,.
\]
Incidentally, we point out that the foregoing reductions of Rankine-Hugoniot conditions also occur at the nonlinear level.

We have left aside the case when $\bar \varphi$ is actually continuous. This may indeed indeed happen but the required changes in the argument are relatively straightforward. Indeed, the comparison to smooth problems turns out to be even slightly simpler in this case since the discontinuity of $\bar \varphi$ is the main obstacle to directly extend to weak solutions on $\R$ the algebraic manipulations carried out in the smooth case.

Now, performing algebraic manipulations as in Subsection~\ref{s:degen} on the corresponding spectral problems, including the introduction of
\[
\tau:=\lambda\,(h-\psi\,\bar h_x)
\] 
(as a function on $\R^*$) and the elimination of $\psi$, reduces the question of the spectral stability of the wave under consideration to $C_f\geq C_t$ and for any $\lambda$ with $\Re(\lambda)>0$, the problem
\begin{align*}
L_{red}(\lambda)\bp \tau\\U\ep&=G&\textrm{on }\R^*\,,\\
\bp
[U]\\
[(g'\bar h+3\bar\varphi\bar h^2)\,\tau+(\hat g\bar h-\bar h^3\bar\varphi_x)\,U]
\ep&=G_0\,,
\end{align*}
is boundedly invertible.

Classical arguments, that may be thought as variations on the one used in Section~\ref{s:stab-smooth}, show that spectral stability (respectively strong spectral stability) is then equivalent to $C_f\geq C_t$, $F_\pm\leq 2$ 
and the non existence, for $\lambda$ such that $\Re(\lambda)>0$ 
of a nonzero $U$ (in the relevant functional space) solving 
\begin{align*}
U''+f_1 U'+(f_2\lambda^2+f_3\lambda+f_4)U&=0&\textrm{on }\R^*\,,\\
[-\bar h(g'\bar h+3\bar\varphi\bar h^2)\,U'
]&=0\,,&
[U]&=0\,.
\end{align*}
In the foregoing, $f_1$, $f_2$, $f_3$ and $f_4$ are as in \eqref{Ueq}, in particular $f_4\equiv 0$.

\subsection{Limit of continuous waves}

Let $\delta:=\bar h-h_0$. Consider a mollified sequence $\delta^\eps$, arising from convolution with a rescaling of a smooth compactly supported kernel, and the corresponding profiles $(\bar h_\eps,c,\bar\varphi_\eps,0)$ with the same limiting height $h_0$. In particular, $\bar h_\eps$ is bounded uniformly in $\eps$ and $\bar h_\eps-h_0$ converges to $\bar h-h_0$ in $L^p(\R)$ for any $1\leq p<\infty$, and in $L^\infty(\R\setminus[-x_0,x_0])$ for any $x_0>0$. It follows from \eqref{weird} that similar conclusions also hold for $\bar\varphi_\eps$. Moreover the foregoing convergences also hold in exponentially weighted topologies (with weights adapted to the convergence rate of $\delta$). 

At this stage, we would like to reformulate eigenvalue problems for both smooth $(\bar h_\eps,c,\bar\varphi_\eps,0)$ and discontinuous $(\bar h,c,\bar\varphi,0)$ so as to reach a form for which continuity with respect to $\eps$ stem from the above convergences and standard arguments. Once this is done, it follows that if for some $\lambda_0$ with $\Re(\lambda_0)>0$ the discontinuous $\lambda_0$-eigenvector system possesses a non trivial solution, this is also true when $\eps>0$ is sufficiently small for the smooth $\lambda_\eps$-eigenvector system associated with some $\lambda_\eps$ converging to $\lambda_0$ when $\eps$ goes to zero. Hence the conclusion.

Such a convenient formulation is obtained by setting (for instance)
\begin{align*}
w_1&=U\,,&
w_2&=\bar h_\eps(g'\bar h_\eps+3\bar\varphi_\eps\bar h_\eps^2)\,U'
-C_f\,|c|\,c\,U\,.
\end{align*}
Indeed, the $\lambda$-eigenvalue problem then takes the form
\begin{align*}
w_1'&=\frac{C_f\,|c|\,c}{\bar h_\eps(g'\bar h_\eps+3\bar\varphi_\eps\bar h_\eps^2)}\,w_1
+\frac{1}{\bar h_\eps(g'\bar h_\eps+3\bar\varphi_\eps\bar h_\eps^2)}\,w_2&\textrm{on }\R^*\,,\\
w_2'&=(\lambda^2\bar h_\eps+\lambda\,2\,C_f\,c)\,w_1&\textrm{on }\R^*\,,\\
[w_1]&=0\,,&
[w_2]=0\,.
\end{align*}

\section{Numerical time-evolution experiments}\label{s:num}
We conclude our study with a series of numerical time-evolution experiments carried
out with the use of the numerical package CLAWPACK \cite{C1,C2}.


Many of the numerical experiments are chosen according to the principle that convective waves are mostly forced by $\varphi$-variations, with $h$-variations being slaved to those through \eqref{weird}, while $(U,\Phi)$ are held constant, $\Phi$ being identically zero. The main reason supporting the discrepancy in our perception of respective roles of $h$ and $\varphi$ is that among the traveling waves we have identified only convective waves may carry out variations in $\varphi$. 

Let us also point out that for other kinds of waves we expect, from both modeling considerations and analysis in \cite{RYZ,RYZ2}, $\Phi$ to remain very small everywhere except near discontinuities. Consistently, in many of our simulations $\Phi$ remains very small throughout the whole time evolution. When this is the case, we omit the panels for variable $\Phi$.

\subsection{Experiments}

\subsubsection*{Figures~\ref{figure2} and~\ref{figure3}}

We begin with a time evolution starting from a Riemann type data, joining two stable equilibria, with initial variations purely in the $\varphi$-variable.

Explicitly, in Figure~\ref{figure2}, we show the result of the simulation of \eqref{rg4}-\eqref{C} for $g'=10\cos(\frac{\pi}{10})$, $\hat{g}=10\sin(\frac{\pi}{10})$, $C_t=0.9$, and $C_f=1$ and with initial data $h\equiv h_0=1$, $U\equiv c=\sqrt{h_0\hat{g}/C_f}$, $\Phi\equiv 0$ and $\varphi=\varphi_L\mathbbm{1}_{x\le 50}+\varphi_R\mathbbm{1}_{50<x}$ where $\varphi_L=0.2$ and $\varphi_R=0.5$. From left to right, we show the solutions at $t=0$, $1$, $10$, and $95$, respectively, with panels for variable $\Phi$ omitted. 

\begin{figure}[htbp]
\begin{center}
\includegraphics[scale=.42]{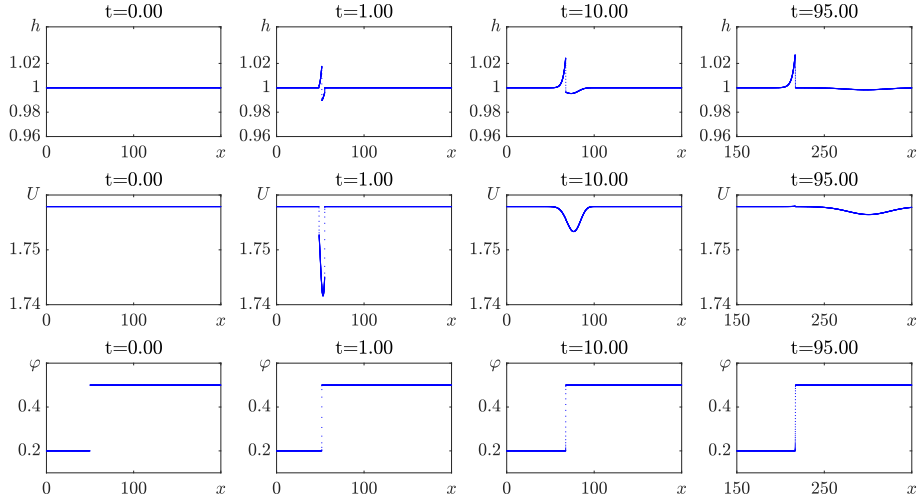}
\end{center}
\caption{Contact discontinuity (stable constant state). The simulation shows emergence of a traveling convective wave.}
\label{figure2}
\end{figure}

Up to numerical smoothing effects, Figure~\ref{figure2} exhibits the emergence of a convective wave profile traveling at speed $c$, with $\varphi$ keeping its initial step-like shape, and $h$ constant equal to $h_0$ at the right of the discontinuity. The variable $U$ undergoes significant variations before spreading out back to constant $c$. 

To confirm the observation, in Figure~\ref{figure3}, we compare the solution at time $t=95$ with the only convective wave profile such that $\bar\varphi=\varphi_L\mathbbm{1}_{x\le 217}+\varphi_R\mathbbm{1}_{217<x}$ and $\bar h(x)=h_0$ when $x>217$. We compute a numerical approximation of the latter by first solving the reduced Rankine-Hugoniot condition \eqref{redRH2} so as to determine $h_L$, the left-value at the jump of the height component, and then solving backwards from the jump the ODE \eqref{weird}. More explicitly, $h_L$ is obtained by solving the cubic equation $\varphi_L h_L^3+g'h_L^2/2=\varphi_R h_R^3+g'h_R^2/2$, with $h_R=h_0$, giving $h_L=1.0292\ldots $, whereas the non-constant part of $h$ is then deduced by solving
\be 
\label{fode_varphi_constant}
h'=\frac{\hat{g}(h-h_0)}{g'h + 3\varphi_L h^2}\quad\textrm{on }(-\infty,217],\qquad\qquad
h(217)=h_L\,.
\ee 
Up to numerical smoothing near the discontinuity, the matching is very convincing. 

\begin{figure}[htbp]
\begin{center}
\includegraphics[scale=.42]{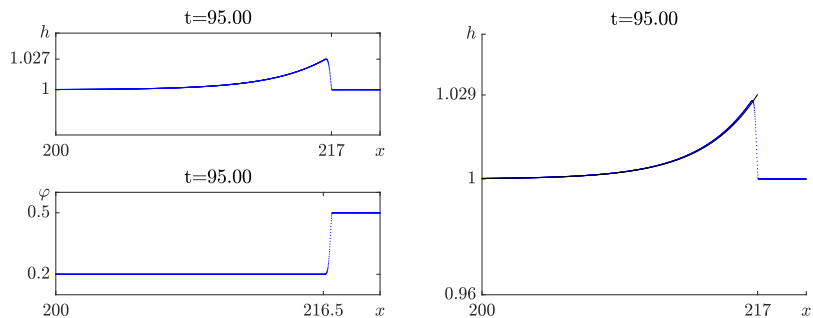}
\end{center}
\caption{Left: Blow up of \ref{figure2} at $t=95$. Right: comparison of analytically computed convective profile (black curve) and simulated convective profile (dotted curve).}
\label{figure3}
\end{figure}

\subsubsection*{Figure~\ref{figure4}}

In Figure~\ref{figure4}, we also provide a time evolution starting from a Riemann type data, joining two equilibria, with initial variations purely in the $\varphi$-variable, but this time the endstates are unstable.

Explicitly, in Figure~\ref{figure4}, we simulate with $g'=10\cos(\frac{\pi}{6})$, $\hat{g}=10\sin(\frac{\pi}{6})$, $C_t=0.04$, and $C_f=0.05$ and with initial data $h\equiv h_0=1$, $U\equiv c=\sqrt{h_0\hat{g}/C_f}$, $\Phi\equiv 0$ and $\varphi=\varphi_L\mathbbm{1}_{x\le 50}+\varphi_R\mathbbm{1}_{50<x}$ where $\varphi_L=0.3$ and $\varphi_R=0.1$. 

In simulations associated with Figures~\ref{figure2} and~\ref{figure3}, the limiting constant states have been chosen to satisfy the hydrodynamical stability condition \eqref{hydro_RG}. When the condition is not satisfied, shocks can form in front of the convective wave. That is, consistent with the discussion of Section \ref{s:full}, in the absence of hydrodynamic stability, the simple time-asymptotic structure predicted by the formal equilibrium system breaks down, and more complicated patterns are expected to emerge. This is what we observe on Figure~\ref{figure4}. 

Note that, as expected, variations of $\Phi$ are located near discontinuities that are not conveyed by convective waves. 

\begin{figure}[htbp]
\begin{center}
\includegraphics[scale=.42]{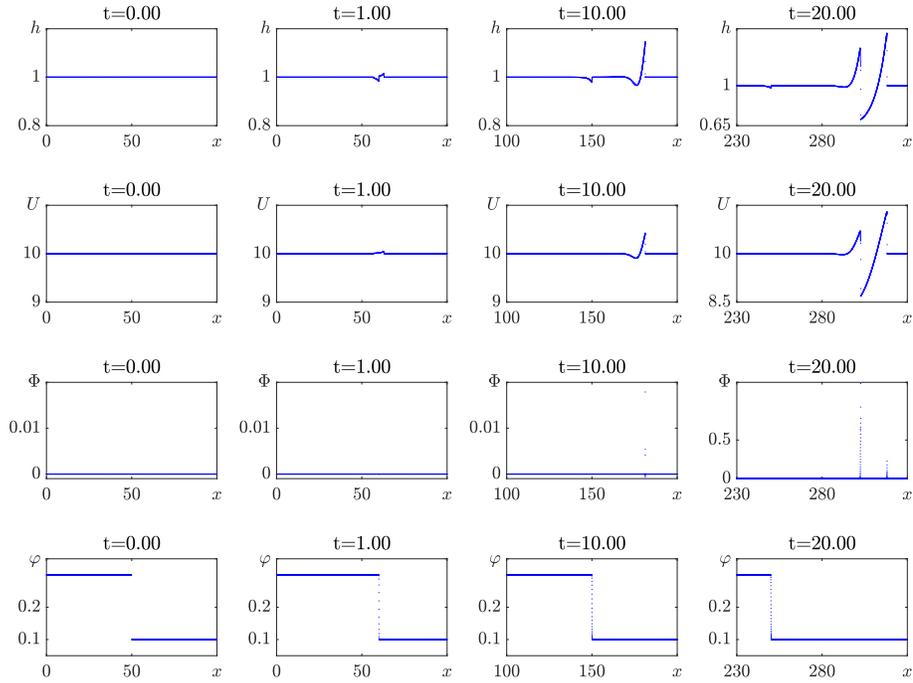}
\end{center}
\caption{Contact discontinuity (unstable constant states). The simulation shows emergence of a traveling convective wave together with shocks in front.}
\label{figure4}
\end{figure}

\subsubsection*{Figures~\ref{figure5} and~\ref{figure6}}

In Figures~\ref{figure5} and~\ref{figure6}, we directly test nonlinear stability of smooth convective waves. Our goal is two-fold. On one hand we want to provide an example on which a relatively small perturbation is resolved in large time into another convective wave with very significantly different shape compared to the initial unperturbed profile. On the other hand, we want to exemplify that perturbations in $\varphi$ have a much stronger impact than those in $h$.

In both figures, parameters are chosen as $g'=10\cos(\frac{\pi}{10})$, $\hat{g}=10\sin(\frac{\pi}{10})$, $C_t=0.8$, and $C_f=1$, and the reference convective wave is obtained by solving \eqref{weird} in $\bar \varphi$ with $h_0=1$, $\bar U\equiv c=\sqrt{h_0\hat{g}/C_f}$, $\bar\Phi\equiv 0$, $\bar h=h_0+\delta$ where

\[
\delta=0.02\times\mathbbm{1}_{|x-3|<1}e^{-\frac{1}{1-(x-3)^2}}\,,
\]
and $\bar{\varphi}(-\infty)=4$, $\kappa=g'h_0^2/2+4h_0^3$.

In figure \ref{figure5}, we display the result of a simulation of \eqref{rg4}-\eqref{C} with the same parameters but initial data $(\bar h+h_{perturbation},\bar U,\bar \Phi,\bar\varphi)$ where
\[
h_{perturbation}=-0.01\times\mathbbm{1}_{|x-6|<1}e^{-\frac{1}{1-(x-6)^2}}\,.
\]
From left to right, we show the solution at $t=0$, $0.2$, $1$, and $6$, respectively,  panels for variable $\Phi$ being omitted. The simulation shows convergence to a traveling convective wave profile having mild difference in all variables when compared with a suitable translate of the unperturbed initial profile $(\bar h,\bar U,\bar \Phi,\bar\varphi)$.
  
\begin{figure}[htbp]
\begin{center}
\includegraphics[scale=.42]{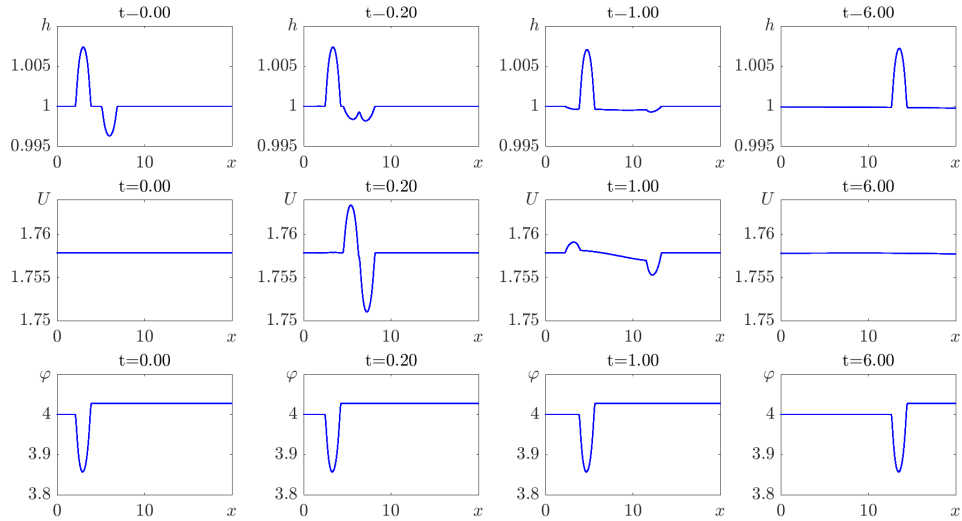}
\end{center}
\caption{Convective wave with perturbation on $h$. The simulation shows convergence to the unperturbed traveling convective wave.}
\label{figure5}
\end{figure}

In figure \ref{figure6}, the computation is completely similar but with initial data 
$(\bar h,\bar U,\bar \Phi,\bar\varphi+\varphi_{perturbation})$ where
\[
\varphi_{perturbation}=0.1\times\mathbbm{1}_{|x-5|<1}e^{-\frac{1}{1-(x-5)^2}}\,.
\]
The numerical outcome shows convergence to another convective wave, whose profile seems to have $\varphi$-component equal to $\bar\varphi+\varphi_{perturbation}$.

\begin{figure}[htbp]
\begin{center}
\includegraphics[scale=.42]{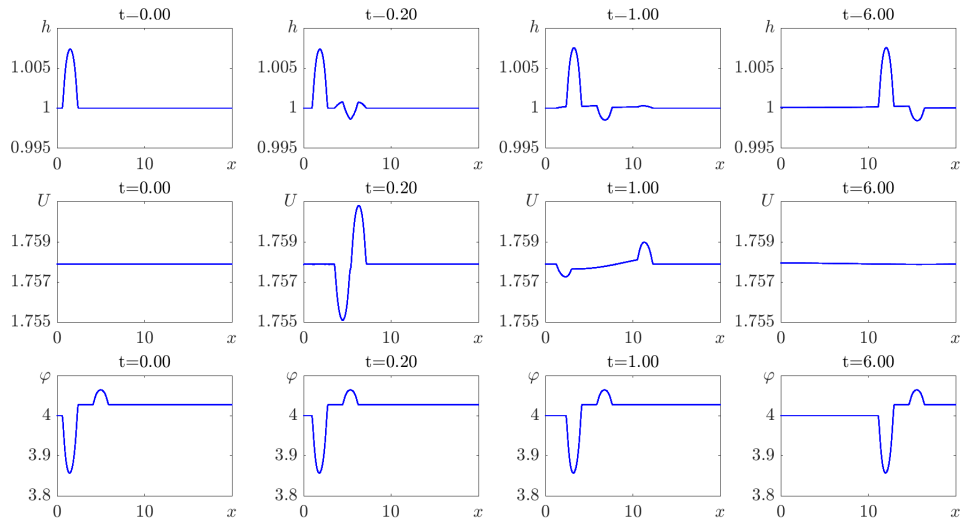}
\end{center}
	\caption{Convective wave with perturbation on $\varphi$. The simulation shows convergence to another convective wave profile whose $\varphi$ component shows mild differences with the \emph{perturbed} initial $\varphi$.}
\label{figure6}
\end{figure}

\subsubsection*{Figures~\ref{figure7} and~\ref{figure8}}

Next, we provide experiments of more direct practical interest, where we investigate interactions of hydraulic jumps arising from a dam break with downward bottom enstrophy. The dam-break experiment is a common test-case for (SV) codes. The $(h,U)$ part of the initial data is chosen to mimic the instantaneous removal of a dam wall separating two fluids at equilibrium with higher height upward. In our numerical tests we add some bottom vorticity ahead of the dam-break location.

We run these numerical experiments with unstable equilibria, but, moreover, we tune parameters to observe first a convectively stable hydraulic jump and then one that is unstable even in the convective sense. When doing so, we use analytic insights from \cite{RYZ}, where we completely determine conditions for convective stability.

Numerical simulations in Figure~\ref{figure7} have been carried out with $g'=5\sqrt{3}$, $\hat{g}=5$, $C_t=0.04$, and $C_f=0.05$ and initial data given by $\Phi\equiv 0$,
\begin{align*}
h&=h_L\mathbbm{1}_{x\le 5}+h_R\mathbbm{1}_{x>5}\,,\qquad
U\,=U_L\mathbbm{1}_{x\le 5}+U_R\mathbbm{1}_{x>5}\,,\\
\varphi&=0.3\mathbbm{1}_{x\le 10}+0.1\mathbbm{1}_{10<x\le 20}+0.5\mathbbm{1}_{20<x\le 30}+0.2\mathbbm{1}_{30<x\le 40}+0.6\mathbbm{1}_{40<x}\,,
\end{align*}
where $h_L=1$, $h_R=0.2$, $U_L=\sqrt{\hat{g}h_L/C_f}=10$ and $U_R=\sqrt{\hat{g}h_R/C_f}=2\sqrt{5}$. From left to right, we show the solution at $t=0$, $5$, $10$, and $15$, respectively.

Figure~\ref{figure7} exhibits at the end the superposition of a convective wave which carries variations in $\varphi$ and travels at speed $U_L=10$ and a (newly discovered) non-monotone hydraulic shock travels at speed $(25-\sqrt{5})/2$ (hence moving faster than the convective wave). Both the non-monotone hydraulic shock and the convective wave look very stable. The convective wave profile has the same structure (up to numerical smoothing) than the initial $\varphi$ component but a different shape, variations occurring over a shorter spatial interval. Incidentally we point out that this is the experiment reported in Figure~\ref{figure7} that has motivated our investigation of convective stability in the present paper and in its companion \cite{RYZ}, and the corresponding revisitation in \cite{FRYZ} of the Saint-Venant analysis of \cite{YZ,SYZ}.

\begin{figure}[htbp]
\begin{center}
\includegraphics[scale=.42]{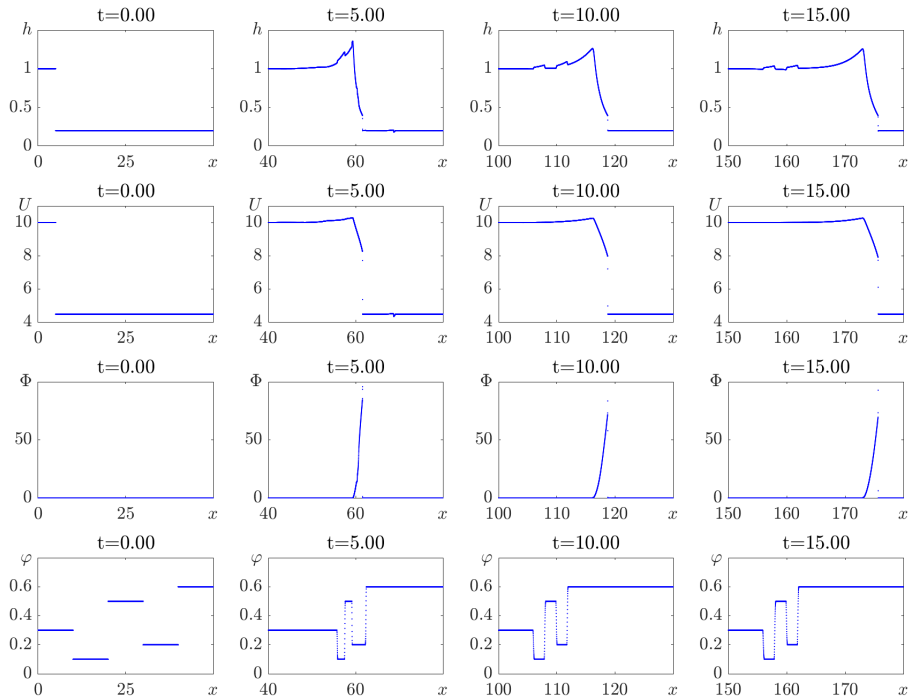}
\end{center}
\caption{Dam-break data in $(h,U)$ with multiple jumps in $\varphi$. The simulation shows asymptotically a non-monotone hydraulic shock to the right with a convection wave trailing behind.}
\label{figure7}
\end{figure}

Figure~\ref{figure8} report simulation of \eqref{rg4}-\eqref{C} with the same choice of parameters and initial data except that $h_R=0.5$ and, accordingly, $U_R=\sqrt{\hat{g}h_R/C_f}=5\sqrt{2}$. The final outcome is however significantly different. As is apparent at $t=10$, on intermediate times we do observe similar phenomena, but afterwards the hydraulic shock starts to be subject to instability and gradually develops a second shock as visible at $t=25$. 

\begin{figure}[htbp]
\begin{center}
\includegraphics[scale=.42]{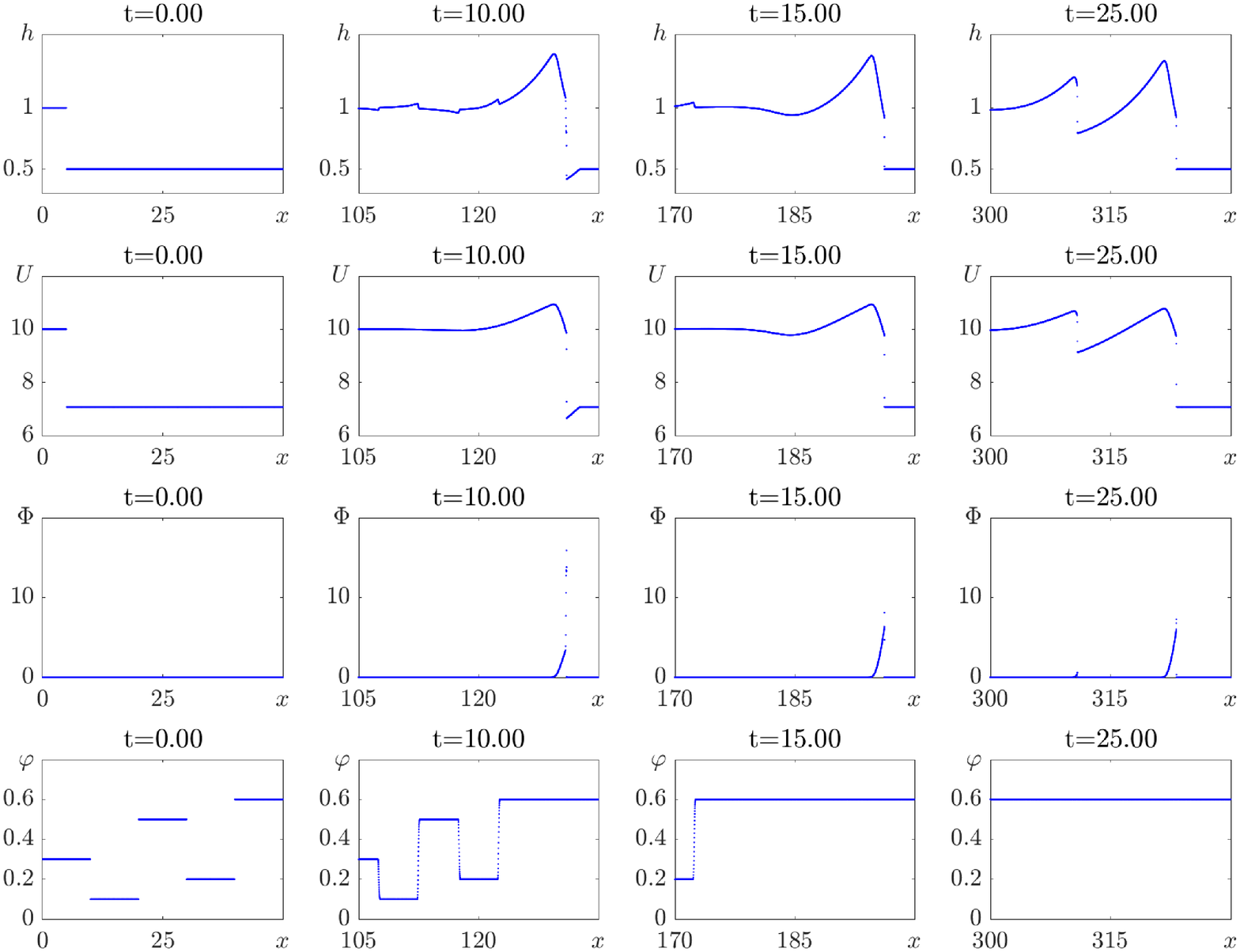}
\end{center}
\caption{Dam-break data in $h/U$ with multiple jumps in $\varphi$. The simulation
shows first show emergence of a non-monotone hydraulic shock to the right with a
convection wave trailing behind. The hydraulic shock is convectively unstable and a second shock emerges  after $t=15.00$. See the plot at $t=25$.}
\label{figure8}
\end{figure}

\subsubsection*{Figures~\ref{figure9} and~\ref{figure10}}

We conclude our numerical investigations with spatially periodic data, a case for which we have very few a priori analytic knowledge. Though we are mostly interested in localized  (hence non-periodic) perturbations of periodic backgrounds, for numerical reasons we provide numerical simulations with periodic boundary conditions. Our goal is to observe whether a pure $\varphi$ periodic perturbation on top of a constant equilibria is resolved into a stable periodic convective wave. 

This is exactly what is happening in the experiment reported in Figure~\ref{figure9}. The latter has been carried out for $g'=10\cos(\frac{\pi}{10})$, $\hat{g}=10\sin(\frac{\pi}{10})$, $C_t=0.9$, and $C_f=1$, with initial data $h\equiv h_0=1$, $\varphi=2+\sin(\pi x)$, $U\equiv c=\sqrt{h_0\hat{g}/C_f}$, $\Phi\equiv 0$, and periodic boundary conditions over $[0,10]$. From left to right, we show the solutions at $t=0$, $0.5$, $2$, and $5$, respectively, $\Phi$ panels being omitted.

\begin{figure}[htbp]
\begin{center}
\includegraphics[scale=.42]{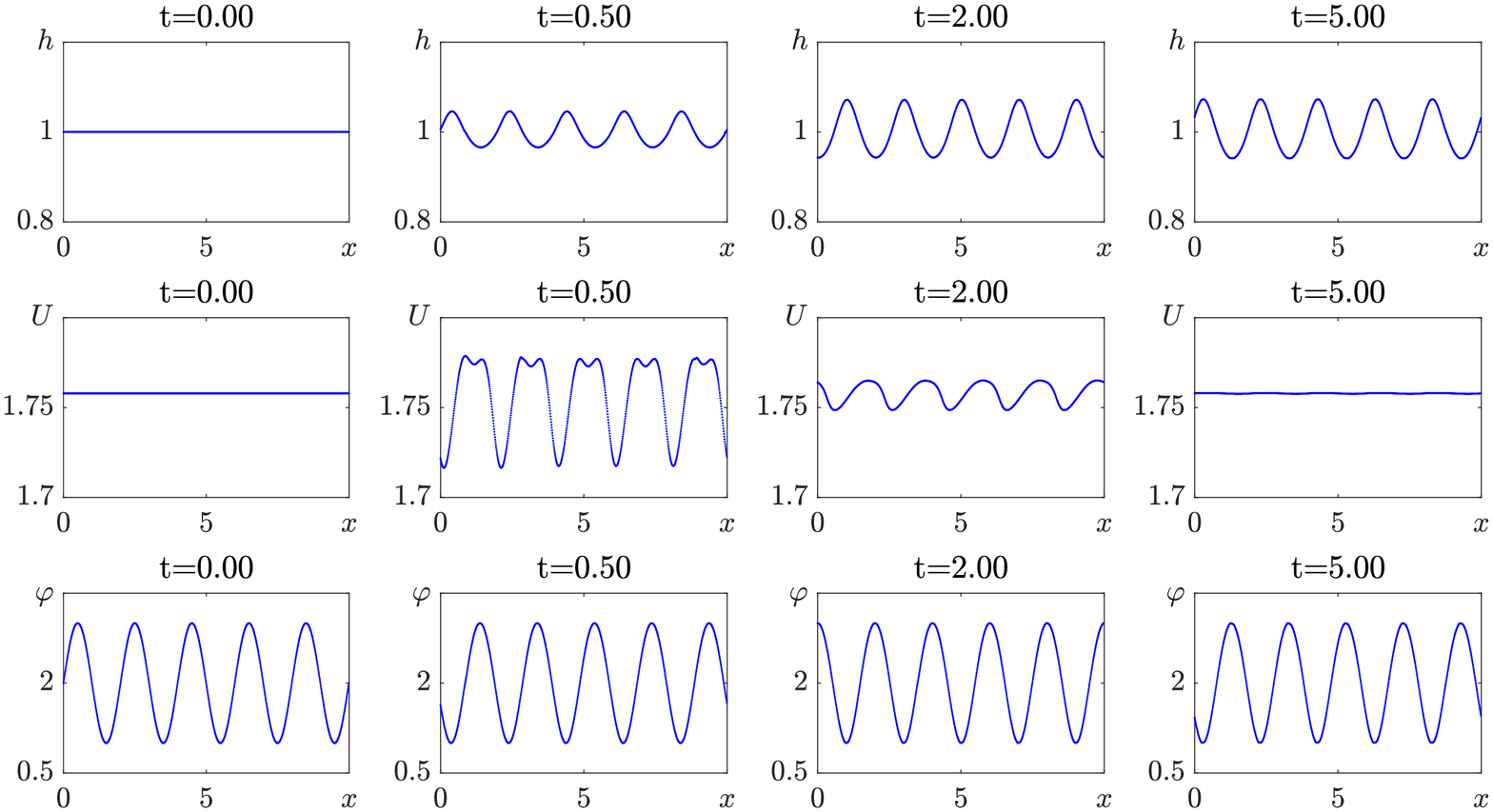}
\end{center}
\caption{Periodic initial data in $\varphi$. The simulation shows convergence to a periodic traveling convective wave.}
\label{figure9}
\end{figure}

Figure~\ref{figure10} shows the result of the simulation of \eqref{rg4}-\eqref{C} for $g'=10\cos(\frac{\pi}{6})$, $\hat{g}=10\sin(\frac{\pi}{6})$, $C_t=0.04$, and $C_f=0.05$, with initial data $h\equiv h_0=1$, $\varphi=2+\sin(\pi x)$, $U\equiv c=\sqrt{h_0\hat{g}/C_f}$, $\Phi\equiv 0$, and periodic boundary conditions on $[0,10]$.  From left to right, we show the solutions at $t=0$, $5$, $10$, and $40$, respectively.

The corresponding final outcome seems well described as a time-quasiperiodic solution resulting from the superposition of two waves traveling at different speeds\footnote{The question of whether the quasiperiodic object is actually a periodic one depends on whether or not the two speeds are rationally related, a question undecidable from numerical observations.}, namely one smooth periodic convective wave and one roll wave. In the end, variations of $U$ and $\Phi$ are conveyed by the roll-wave part whereas $\varphi$ variations are supported by the convective part. On the $h$ part one does see the quasiperiodic superposition, with the position of the roll-wave part being noticeable on discontinuities. 

\begin{figure}[htbp]
\begin{center}
\includegraphics[scale=.42]{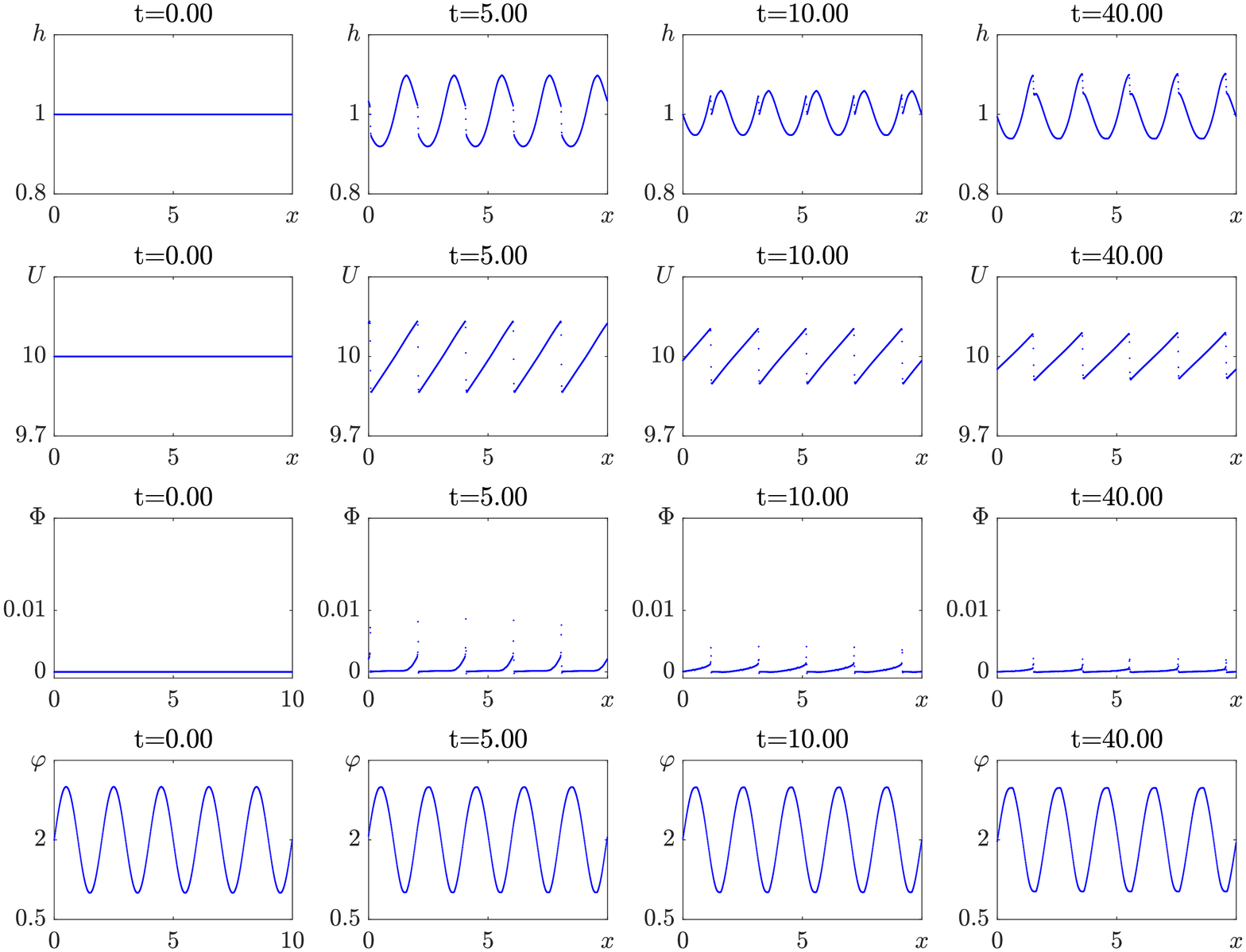}
\end{center}
\caption{Periodic initial data in $\varphi$. The simulation shows superposition of a traveling convective wave and a discontinuous roll wave.}
\label{figure10}
\end{figure}

\subsection{Numerical conclusions}\label{s:numc}
As noted in Section \ref{s:eq4}, it is readily seen that equilibrium system \eqref{eq2} admits a unique solution of the Riemann problem between two equilibrium states, consisting of a shock followed by a contact discontinuity. The result of our experiments suggest that in the hydrodynamically stable case that \eqref{hydro_RG} is satisfied both for the initial equilibrium states and the states predicted by the equilibrium Riemann solution, Riemann data for the full system  \eqref{rg4} yields time-asymptotic behavior conforming to this prediction, but with relaxation profiles substituted for (equilibrium) shock and contact discontinuities. However, in the hydrodynamically unstable case that \eqref{hydro_RG} fails, either at endstates or intermediate states in the Riemann solution for the equilibrium system, more complicated but still coherent time-asymptotic patterns can emerge, involving nonmonotone shock profiles and various kinds of composite waves involving combinations of shock and roll waves.

Though we did not investigate it systematically, our brief study of the periodic case suggests that still more complicated possibilities may occur, with approximate superposition of periodic waves moving with different speeds. We conjecture that this may indeed result in both time-periodic and quasiperiodic solutions that are not traveling waves. The rigorous study of their existence would be an extremely interesting direction for further study.


\section{Discussion and open problems}\label{s:disc}
We have shown the appearance in the Richard-Gavrilyuk model (RG) of a new type of convective waves not present in the standard Saint-Venant model (SV), and characterized them as linearly degenerate waves analogous to those occurring in the entropy field for gas dynamics. Indeed, they correspond to one of the two enstrophy fields introduced in (RG), namely bottom vorticity $\varphi$.

Together with hydraulic shocks, these waves play an important role in time-asymptotic behavior
from asymptotically constant initial data, 
corresponding to relaxation profiles of contact discontinuities in the associated equlibrium system \eqref{eq2} in the same way that hydraulic shocks both in (SV) and in (RG) correspond to relaxation profiles of Lax-type equilibrium shocks \cite{YZ,RYZ}.
However, different from contact relaxation profiles studied in \cite{HPW,Zh}
under Kawashima-type dissipativity conditions, these do not spread diffusively into a universal Gaussian-like profile, but appear to persist unchanged for all time; accordingly, there appear to be an infinitude of different ultimate time-asymptotic profiles.

Our numerical time evolution experiments indicate that both components ---
hydraulic shocks and 
asymptotically constant 
convective waves --- are 
stable when their endstates are, 
with arbitrary data resolving eventually into a noninteracting convective wave/hydraulic shock pattern. Under the hydrodynamic stability conditions \eqref{hydrodynamicstability} on endstates, we have shown analytically, by Sturm-Liouville considerations, that 
asymptotically-constant convective waves are spectrally stable. The related problem of existence and spectral stability of hydraulic shocks is studied in \cite{RYZ}.

Interestingly, stability appear also to hold sometimes even for contact-shock components arising from asymptotically-constant data with endstates violating the hydrodynamic stability conditions \eqref{hydrodynamicstability}. 
This 
can be understood as an example of {\it convective stability}, whereupon instabilities arising at the end state are convected into the wave and unltimately stabilized. Indeed, one may verify that instabilities arising through violation of  \eqref{hydrodynamicstability} 
at the left endstate 
correspond to the genuinely nonlinear $h$-field, which travels transverse to contact waves, and ultimately into the leading Lax shock, thus stabilizing. This phenomenon may be captured by working in a weighted norm, amounting to precisely the Sturm-Liouville transformation studied earlier. Thus, convective spectral stability of convective (in different sense) wave solutions follows for 
%
asymptotically-constant 
profiles by the same Sturm-Liouville argument as in the hydrodynamically stable case.

For both hydraulic shocks and convective-wave solutions, nonlinear stability remains a very interesting open question. For, independent of Froude number $\tilde F$ and strict hydrodynamic stability  in the $h$ equilibrium mode, the fact that the convective $\varphi$ mode is a characteristic mode of both relaxation system (RG) and equilibrium system \eqref{eq2} means that Kawashima's genuine coupling condition is violated for this model,
leading, independent of $\tilde F$, to {\it neutral instability} in the $\varphi$ equilibrium mode, corresponding with the infinite-dimensional kernel of the linearized operator about the wave, as seen in Section \ref{s:stab-smooth}. Thus, nonlinear 
asymptotic 
stability, if valid, is in the weak sense of 
stability of the infinite-dimensional family of convective waves in its entirety and not of individual waves, or of some of their finite-dimensional orbits under the action of symmetry group. 
It may also be that waves are merely \emph{metastable}, with slowly growing instabilities occurring through resonance with this nondecaying infinite-dimensional kernel.

The resolution of this issue of nonlinear stability is an extremely interesting direction for future study. 
Every example we know of among the relatively few analytical results analyzing the large-time dynamics near an infinite-dimensional family of (relative) equilibria --- including amazingly impressive analyses near homogeneous distributions of Vlasov-Poisson systems \cite{Mouhot-Villani,Bedrossian-Masmoudi-Mouhot} and shear flows of incompressible Euler systems \cite{Bedrossian-Masmoudi} --- requires an extremely fine understanding of the nonlinear structure of the system at hand. 
Indeed, it appears to be a challenging problem even for the simplest case of stability of
{\it constant, equilibrium solutions}, similar to the study \cite{LZ} of compressible Navier-Stokes equations with zero heat conductivity, 
or, even closer, to the analysis of a class of degenerate hyperbolic systems with relaxation in \cite{Mascia-Natalini}. For comparison with the latter, we point out that considering convective stabilization by spatial weights of discontinuous waves could also help preventing the formation of new singularities (hence relaxing the linear degeneracy assumption of \cite{Mascia-Natalini} for the marginally stable mode), as exemplified in \cite{Blochas-Wheeler}.
%

For 
periodic convective waves, even spectral stability appears to be more complicated.
Another very interesting direction for further investigation is the numerical study of 
the latter. 
In particular, considering the example of (acoustic, or genuinely nonlinear) periodic roll waves for (SV), which arise if and only if the Froude number $F$ is $>2$, violating hydrodynamic stability, one may ask what is the role of the Froude number 
in stability of periodic convective waves. Though it does not affect existence of periodic convective waves as in the case of periodic acoustic waves, it might yet potentially affect stability. Whether or not this is the case is a natural question that would be interesting to resolve. The related problem of existence and spectral stability of roll waves is studied in \cite{RYZ2}.

Even more generally, our numerical experiments of spatially periodic solutions suggest that the (RG) model also supports time-quasiperiodic solutions, seemingly obtained by superposing convective waves and roll waves, whose study is widely open from any point of view. 

\appendix

\section{Hydrodynamic stability and extended consistent splitting}\label{app:split}

We detail here a few computations related to the near-constant stability analysis.

\subsection{Hydrodynamic stability}

We first recover \eqref{hydroRGF}--\eqref{hydroRGC} as criteria for the stability of a constant equilibrium. The reduction process of Subsection~\ref{s:degen} applies equally well to the constant-coefficient analysis so that we may focus on the invertibility of $L_{red}^{\varphi_0}(\lambda)$ defined by
\begin{align*}
L_{red}^{\varphi_0}(\lambda)\bp \tau\\U\ep
&:=\bp1&h_0\d_x\\
\left(g'+3h_0\varphi_0\right)\d_x
-h_0^{-1}\,\hat{g}
&\lambda^2+\lambda\,2C_fc\,h_0^{-1}\ep\bp \tau\\U\ep\,.
\end{align*}

By either Fourier or Green-function arguments, one readily checks that $L_{red}^{\varphi_0}(\lambda)$ is invertible (on a reasonable weightless Sobolev space) if and only if there is no non-zero trigonometric monomial $(\tau, U)$ such that $L_{red}^{\varphi_0}(\lambda)(\tau, U)^T=0$. This is equivalent to the fact that the matrix 
\begin{align*}
M^{\varphi_0}(\lambda)
&:=
\bp 0&h_0\\
\left(g'+3h_0\varphi_0\right)
&0\ep^{-1}
\bp -1&0\\
h_0^{-1}\,\hat{g}
&-\lambda^2-\lambda\,2C_fc\,h_0^{-1}\ep\\
&\,=\,
\frac{1}{h_0\,\left(g'+3h_0\varphi_0\right)}
\bp
\hat{g}
&-h_0\lambda^2-\lambda\,2C_fc\,\\
-\left(g'+3h_0\varphi_0\right)&0\ep
\end{align*}
has no purely-imaginary eigenvalue. These eigenvalues are readily computed to be 
\begin{align*}
\gamma_{1}^{\varphi_0}(\lambda)
&=
\frac{\hat{g}+\sqrt{\hat{g}^2+4\,\left(g'+3\,h_0\,\varphi_0\right)\,\left(2\,C_f\,c\,\lambda+h_0\,\lambda ^2\right)}}{2\,h_0\,\left(g'+3\,h_0\,\varphi_0 \right)}\,,
\\
\gamma_{2}^{\varphi_0}(\lambda)&=
\frac{\hat{g}-\sqrt{\hat{g}^2+4\,\left(g'+3\,h_0\,\varphi_0\right)\,\left(2\,C_f\,c\,\lambda+h_0\,\lambda ^2\right)}}{2\,h_0\,\left(g'+3\,h_0\,\varphi_0\right)}\,.
\end{align*}
Note that we have implicitly excluded cases when $\hat{g}^2+4\,\left(g'+3\,h_0\,\varphi_0\right)\,\left(2\,C_f\,c\,\lambda+h_0\,\lambda ^2\right)\in (-\infty,0)$ so as to pick a determination of $\sqrt{\cdot}$, but this is irrelevant for the stability analysis, since a direct computation shows that $2\,C_f\,c\,\lambda+h_0\,\lambda ^2\in (-\infty,0)$ implies $\Re(\lambda)<0$.

By taking it account the sign of the real parts of these spatial eigenvalues when $\Re(\lambda)\gg 1$, one deduces that $L_{red}^{\varphi_0}(\lambda)$ is invertible for any $\lambda\neq0$ such that $\Re(\lambda)\geq0$ if and only if for any such $\lambda$,
\[
\Re(\gamma_{1}^{\varphi_0}(\lambda))>0>\Re(\gamma_{2}^{\varphi_0}(\lambda))\,.
\]
More directly this is equivalent to for any $\xi\in\RR$, if $i\xi$ is an eigenvalue of $M^{\varphi_0}(\lambda)$, then $\Re(\lambda)<0$ or $\lambda=0$.

In turn, the condition that $i\xi$ is an eigenvalue of $M^{\varphi_0}(\lambda)$ is equivalent to
\[
\xi^2\,h_0\,\left(g'+3\,h_0\,\varphi_0 \right)
+i\,\xi\,\hat{g}
+\,\frac{2\,C_f\,c}{h_0}\,\lambda+\,\lambda^2\,=\,0
\]
which is solved as
\[
\lambda\,=\,-\frac{C_f\,c}{h_0}\pm\sqrt{C_f^2\,c^2\,h_0^{-2}-i\,\xi\,\hat{g}-\xi^2\,h_0\,\left(g'+3\,h_0\,\varphi_0 \right)}\,.
\]
Note that one of the solutions is zero if and only if $\xi=0$ and in this case the other one is real negative. Setting $m(\xi):=C_f^2\,c^2\,h_0^{-2}-i\,\xi\,\hat{g}-\xi^2\,h_0\,\left(g'+3\,h_0\,\varphi_0 \right)$, one derives that for $\xi\neq0$, both solutions have negative real parts if and only if $\Re(\sqrt{m})< C_f\,c\,h_0^{-1}$, that is, if and only if
\begin{align*}
\sqrt{\frac12\,\left(|m|+\Re(m)\right)}\,&<C_f\,c\,h_0^{-1}\,,&
\textrm{i.e.}&&\sqrt{\Re(m)^2+\Im(m)^2}
\,&<2\,C_f^2\,c^2\,h_0^{-2}-\Re(m)\,,
\end{align*}
hence, if and only if
\begin{align*}
\Im(m)^2\,&<\,4\,C_f^2\,c^2\,h_0^{-2}\,\left(C_f^2\,c^2\,h_0^{-2}-\Re(m)\right)\,,&
\textrm{i.e.}&&
\hat{g}^2
&<\,4\,C_f^2\,c^2\,h_0^{-2}\,h_0\,\left(g'+3\,h_0\,\varphi_0\right)\,.
\end{align*}
Since $c^2=\hat{g}\,h_0/C_f$, the latter condition is indeed, as announced,
\be 
\label{hydro_RG}
\frac{c^2}{h_0\,\left(g'+3\,h_0\,\varphi_0\right)}\,<\,4.
\ee 

\subsection{Extended consistent splitting}

We now discuss consequences of the foregoing computations for the stability of smooth convective waves connecting $(h_0, c, 0,\varphi_-)$ at $-\infty$ to $(h_0, c, 0, \varphi_+)$ at $+\infty$. In unweighted topologies, computations are directly applicable. 

We would like to stress the consequence of using a weight that behaves as $e^{\theta_-\,x}$ near $-\infty$ and as $e^{\theta_+\,x}$ near $+\infty$. In the corresponding weighted topology, one needs that for any $\lambda\neq0$ such that $\Re(\lambda)\geq0$, $L_{red}^{\varphi_-}(\lambda)$ is invertible in spaces weighted with $e^{\theta_-\,x}$, and $L_{red}^{\varphi_+}(\lambda)$ is invertible in spaces weighted with $e^{\theta_+\,x}$. This is equivalent to for any such $\lambda$
\begin{align*}
\Re(\gamma_{1}^{\varphi_+}(\lambda))&>-\theta_+
>\Re(\gamma_{2}^{\varphi_+}(\lambda))\,,&
\Re(\gamma_{1}^{\varphi_-}(\lambda))&>-\theta_-
>\Re(\gamma_{2}^{\varphi_-}(\lambda))\,.
\end{align*}
Taking the limit $\lambda\to0$ shows that necessarily $\theta_+\leq 0$ and $\theta_-\leq 0$. Moreover choices
\[
\theta_\pm=
-\frac{\hat{g}}{2\,h_0\,\left(g'+3\,h_0\,\varphi_\pm \right)}
=\frac12f_1(\pm\infty)
\]
fit the requirement independently of any constraint on Froude numbers. This justifies the claims in Subsection~\ref{s:consubs}.

\bibliographystyle{alphaabbr}
\bibliography{Ref-roll}

\end{document}